\documentclass[preprint,11pt]{elsarticle}
\usepackage[T1]{fontenc}
\usepackage{lmodern}
\usepackage[a4paper,margin=27mm]{geometry}
\usepackage{amsmath,amssymb,amsthm,mathtools,enumitem,needspace,etoolbox}
\usepackage[hypertexnames=false,colorlinks=true,linkcolor=black,citecolor=black,urlcolor=blue]{hyperref}
\hypersetup{pdftitle={Finite comparison and countable injective hulls},pdfauthor={Shaohui Peng and Guocheng Dai},pdflang={en},pdfsubject={Bergman Questions 3.1 and 3.2, finite stabilization, and the semiprimary Faith conjecture}}
\journal{Journal of Algebra}
\makeatletter
\def\ps@pprintTitle{\let\@oddhead\@empty\let\@evenhead\@empty\def\@oddfoot{\hfil Revised manuscript, version 14, 30 October 2026\hfil}\let\@evenfoot\@oddfoot}
\makeatother
\allowdisplaybreaks
\setlist[enumerate]{label=\textup{(\roman*)},itemsep=3pt,topsep=4pt}
\newtheorem{theorem}{Theorem}[section]
\newtheorem{proposition}[theorem]{Proposition}
\newtheorem{lemma}[theorem]{Lemma}
\newtheorem{corollary}[theorem]{Corollary}
\theoremstyle{definition}\newtheorem{example}[theorem]{Example}
\theoremstyle{remark}\newtheorem{remark}[theorem]{Remark}
\theoremstyle{definition}\newtheorem{problem}[theorem]{Problem}
\newtheorem{definition}[theorem]{Definition}
\BeforeBeginEnvironment{lemma}{\Needspace{6\baselineskip}}
\BeforeBeginEnvironment{theorem}{\Needspace{5\baselineskip}}
\BeforeBeginEnvironment{proposition}{\Needspace{4\baselineskip}}
\DeclareMathOperator{\Hom}{Hom}
\DeclareMathOperator{\End}{End}
\DeclareMathOperator{\Ext}{Ext}
\DeclareMathOperator{\Tor}{Tor}
\DeclareMathOperator{\Soc}{Soc}
\newcommand{\N}{\mathbb N}
\newcommand{\id}{\mathrm{id}}
\DeclareMathOperator{\Jac}{Jac}
\DeclareMathOperator{\im}{Im}
\newcommand{\C}{\mathcal C}
\newcommand{\ann}{\operatorname{ann}}
\newcommand{\la}{\operatorname l}
\newcommand{\ra}{\operatorname r}
\begin{document}
\begin{frontmatter}
\title{Finite comparison and countable injective hulls}
\author[aff1]{Shaohui Peng}
\ead{20250801074@stu.scinu.edu.cn}
\author[aff1]{Guocheng Dai}
\ead{ggcdai@163.com}
\address[aff1]{School of Mathematical Sciences, Sichuan Normal University,
Chengdu 610066, P.R. China}
\begin{abstract}
We give a negative answer to the unrestricted product version of
Bergman's finite-comparison theorem, even when all factors have
finite length. Finite comparison does hold for quasi-injective
common modules under finite cogeneration. We also prove a
semidirect-product case of Bergman's Question 3.2.
For a semiprimary right self-injective ring, failure of the
quasi-Frobenius conclusion produces a cyclic module with simple
top and countably infinite homogeneous socle inside a fixed
countable injective hull. We identify its extension class and
an associated opposite-side extension obstruction.
Over the square-zero algebra $k\oplus k^{(\mathbb N)}$, we
compute the countable hull quotient as endomorphisms modulo
finite-rank maps. This quotient also describes the socle-radical
quotient of the regular product, yields a second counterexample
with free cyclic family factors, and detects splitting of
specified simple submodules. The general Faith conjecture remains
unresolved.
\end{abstract}
\end{frontmatter}

\section{Introduction}
Bergman proved that two infinite direct-sum decompositions of a
module admit a finite comparison when the distinguished factors
and one family of summands are finitely generated
\cite[Theorem 1.3]{Bergman}. We study the corresponding product
question, a semidirect-product version, and the extension problem
arising from a countable direct sum inside its injective hull.

\subsection{Finite comparison of products}
For product decompositions
\begin{equation}\label{intro-products}
A\times\prod_{i\in\Lambda}C_i\cong M\cong
B\times\prod_{j\in\Delta}D_j,
\end{equation}
Bergman's Question 3.1 asks for a conclusion of the form
\begin{equation}\label{intro-comparison}
A\oplus Y\cong B\oplus\prod_{j\in\Delta_0}D_j,
\qquad Y\mid\prod_{i\in\Lambda_0}C_i,
\end{equation}
with finite $\Lambda_0$ and $\Delta_0$.
Theorem~\ref{counterexample} gives cyclic length-two modules
$E_n$ and a simple module $L$ such that, for $P=\prod_nE_n$,
$L\oplus P\cong L^2\oplus P$, while
\eqref{intro-comparison} fails for every finite choice and
every allowed summand $Y$. All factors are both finitely
generated and finitely cogenerated.

If the common module is quasi-injective, finite cogeneration
of $A,B,D_j$ suffices. The same conclusion holds when $M$
is nonsingular and these factors have finite uniform dimension
(Theorem~\ref{product-comparison}). Finite projections and
Warfield's cancellation theorem \cite[Corollary 2.6]{Warfield}
give the proof. Example~\ref{prufer-boundary} explains why
direct finiteness alone does not suffice.

For Question 3.2, the small-object and ordinary group cases
are already discussed by Bergman. We give their categorical
formulation and a group-reduct version using the modular law
of Crawley--J\'onsson \cite[Lemma 2.3]{CJ}.
Theorem~\ref{thm:semidirect} treats
\[
M=A\rtimes\bigoplus_i C_i=B\rtimes\bigoplus_jD_j,
\qquad A,B\trianglelefteq M,
\]
where $A,B,D_j$ are finitely generated and the complements
are restricted direct products. It gives an internal finite
semidirect comparison with the actions inherited from $M$.

\subsection{The countable extension in Faith's conjecture}
Write $J=\Jac R$. In the ring-theoretic part we assume
\begin{equation}\label{F}\tag{F}
D=R/J\text{ is semisimple artinian},\qquad
J^d=0,\qquad R_R\text{ is injective},
\end{equation}
where $d\geq1$. The semiprimary form of Faith's conjecture
asks whether every such ring is quasi-Frobenius (QF).
Retain the canonical inclusion $E=R^{(\N)}\subseteq R^{\N}$,
embed its injective hull $I$ in that product, and put $B=I/E$.
The QF conclusion is equivalent to $B=0$.

If $R$ is not QF, Theorem~\ref{homogeneous} constructs
a cyclic $U\subseteq I$ and simple right modules $Q,T$ with
\[
U/UJ\cong Q,\qquad UJ=\Soc U=U\cap E\cong T^{(\N)},
\qquad UJ^2=0.
\]
It also identifies the nonzero Baer class represented by
$0\to E\to E+U\to Q\to0$.
Proposition~\ref{socle-ext-image} computes the image of the
map on $\Ext^1$ induced by $S^{(\N)}\hookrightarrow E$,
where $S=\Soc R$. The same countable quotient of $J/J^2$
gives the left extension obstruction in
Proposition~\ref{opposite-second-layer-test}.

The second-socle pairing is classical
\cite[proof of Proposition 4]{NY97}. Our construction retains
the fixed hull and the coordinate maps representing its
extension class. Ara--Nicholson--Yousif \cite{ANY} study
compatible bimodule pairings, and Iovanov \cite{Iov} compares
cardinalities of Loewy layers. The conditional transfer in
Theorem~\ref{action-criterion} and the dimension argument in
Proposition~\ref{central-layer-criterion} describe two uses of
these ideas. The uniform finite-support assertion in
Problem~\ref{remaining-gap} is not proved here.

\subsection{An explicit connection}
Section~\ref{nonopen-section} applies the countable construction
to the square-zero ring of Theorem~\ref{counterexample}, with
injective coefficient $C=\Hom_k(R,k)$. For its fixed sum and hull,
$E_C=C^{(\N)}\subseteq I_C\subseteq C^{\N}$, we obtain
\[
\frac{\Soc(R^{\N})}{R^{\N}J}
 \cong\frac{I_C}{E_C}
 \cong\Ext_R^1(R/J,E_C)
 \cong\frac{\End_k(k^{(\N)})}{\text{finite-rank maps}}.
\]
This gives a second product counterexample, with all family
factors free cyclic. The length-two product $P$ also enters
the same construction: inclusion of $P\cap I_C$ in $I_C$
induces an injection $\Soc P/PJ\hookrightarrow I_C/E_C$.
Thus the correspondence is realized inside a common product.
The base ring in this example is not self-injective; it is
an application of the countable construction with a different
injective coefficient.

Sections~\ref{restriction-section}--\ref{quotient-section}
give the socle-restriction and homological facts used in the
construction. Section~\ref{homogeneous-section} treats the
two-layer witness and the opposite extension problem.
The structural results on PF rings, perfect rings and injective
hulls used there are classical \cite{Bass,JW,FU,Levine}.

Rings have identity, and modules are unital right modules
unless a side is specified. We write $X^{(\Lambda)}$ and
$X^\Lambda$ for sums and products, and $Y\mid X$ for a module
isomorphic to a direct summand of $X$. An assertion that a
specified inclusion splits refers to that inclusion.
A module is finitely cogenerated if its socle is essential
and has finite length, equivalently if every family of
submodules with zero intersection has a finite subfamily
with zero intersection. Hypothesis \eqref{F} is used only
where stated.

\section{Bergman's product question}\label{product-section}
We begin with the failure of unrestricted finite comparison.
The positive results that follow separate two requirements:
finite projections and cancellation in a common quasi-injective module.

\subsection{A length-two counterexample}\label{counterexample-section}
Let $k$ be a field. Form the commutative ring
\[
R=k[x_1,x_2,\ldots]/(x_ix_j:i,j\geq1),\qquad
J=\bigoplus_{n\geq1}kx_n,
\qquad L=R/J.
\]
For $n\geq1$, put
$E_n=R/\bigoplus_{m\ne n}kx_m$ and $P=\prod_{n\geq1}E_n$.
The module $E_n$ has basis $u_n,v_n$ with
\[
u_nx_m=\delta_{mn}v_n,\qquad v_nx_m=0.
\]
It is cyclic and uniserial of length two, with
$\Soc E_n=E_nJ=kv_n\cong L$.

\begin{theorem}\label{counterexample}
For these modules, $P\cong L\oplus P$. If $a\ne b$ are
finite nonnegative integers, then for all finite
$I_0,J_0\subseteq\N$ and every summand $Y\mid\prod_{n\in I_0}E_n$,
\begin{equation}\label{finite-failure}
L^a\oplus Y\not\cong L^b\oplus\prod_{n\in J_0}E_n.
\end{equation}
In particular, taking $A=L$, $B=L^2$, and $C_n=D_n=E_n$
gives product decompositions with no conclusion of the form
\eqref{finite-comparison}, although every displayed factor
is nonzero and has finite length.
\end{theorem}
\begin{proof}
Write $F=k^{\N}$ and $F_0=k^{(\N)}$. In the coordinates
$P=F\oplus F$ supplied by the bases $u_n,v_n$, the action is
\begin{equation}\label{product-action}
(a,b)\left(c+\sum d_nx_n\right)
 =\left(ca,cb+(d_na_n)_n\right),
\end{equation}
where only finitely many $d_n$ are nonzero. Hence
\begin{equation}\label{radical-discrepancy}
PJ=0\oplus F_0,\qquad \Soc P=0\oplus F.
\end{equation}
For the first equality, apply $\sum b_nx_n$ to
$((1,1,\ldots),0)$ for a finitely supported $(b_n)$.
The second equality follows by testing multiplication by
each $x_n$.

Choose a vector-space complement $F=F_0\oplus W$. Then
\[
P=(F\oplus F_0)\oplus(0\oplus W)
\]
as $R$-modules, and $J$ annihilates $W$.
The vector space $W\cong F/F_0$ is infinite-dimensional:
indicator functions of pairwise disjoint infinite subsets
of $\N$ are linearly independent modulo $F_0$.
Thus $W\cong k\oplus W$, by shifting a countably infinite
part of a basis. This proves $P\cong L\oplus P$.

For a finite set $F'$, the module $E_{F'}=\prod_{n\in F'}E_n$
satisfies $\Soc E_{F'}=E_{F'}J$. If $E_{F'}=Y\oplus Z$, then $E_{F'}J=YJ\oplus ZJ$, so
$(E_{F'}J)\cap Y=YJ$. Thus intersecting the socle equality
with $Y$ gives $\Soc Y=YJ$.
For a finite-length module $N$, define
\[
d(N)=\dim_k(\Soc N/NJ).
\]
This is defined since $J^2=0$, is additive, and is preserved
by isomorphisms. We have $d(Y)=d(E_{J_0})=0$ and $d(L^a)=a$,
which proves \eqref{finite-failure}. Iterating
$P\cong L\oplus P$ gives the asserted product decompositions.
\end{proof}

\paragraph{Comparison with Bergman's product question.}
Bergman's Question 3.1 retains the finite conclusion of his
Theorem 1.3 for product decompositions. His subsequent discussion
replaces finite generation by finite cogeneration to obtain finite
projected embeddings, but leaves the compatibility of their
composites unresolved. Theorem~\ref{counterexample} excludes
the proposed conclusion for every finite choice and every allowed
summand $Y$. Its factors have finite length and satisfy both
finiteness conditions. The proof uses
$\Soc P/PJ\cong F/F_0$ and the vanishing of the corresponding
quotient on finite blocks and their summands. Positive results
require additional hypotheses, as in
Theorem~\ref{product-comparison}.

\begin{corollary}\label{bergman-answer}
The unrestricted direct-product analogue of
\cite[Theorem 1.3]{Bergman}, as asked in
\cite[Question 3.1]{Bergman}, fails with finitely generated
distinguished factors. It also fails after replacing finite
generation by finite cogeneration. Both failures occur with
all distinguished and family factors nonzero and of finite
length.
\end{corollary}
\begin{proof}
In Theorem~\ref{counterexample}, take $A=L$, $B=L^2$ and
$C_n=D_n=E_n$. All these factors are finitely generated and
finitely cogenerated. The product isomorphism holds, while
\eqref{finite-failure} excludes every possible finite
comparison, including every allowed summand $Y$.
\end{proof}

One can also see directly why a finite projected embedding
does not split. Let $\boldsymbol1=(1,1,\ldots)$ and choose
$\lambda:F\to k$ with $\lambda(F_0)=0$ and
$\lambda(\boldsymbol1)=1$. By \eqref{product-action},
\[
P\longrightarrow L,\qquad(a,b)\longmapsto\lambda(b)
\]
is $R$-linear and retracts the copy of $L$ generated by
$(0,\boldsymbol1)$. Its projection into any nonempty finite
$E_{F'}$ is monic, but its image lies in
$\Soc E_{F'}=E_{F'}J$. Every map $E_{F'}\to L$ kills this
image. Thus the projected monomorphism has no retraction.
In particular, the common module in this example is not
quasi-injective, as also follows from
Theorem~\ref{product-comparison}.

\subsection{Finite projections and classical cancellation}
Finite cogeneration passes to submodules and finite direct sums.
If a finitely cogenerated module embeds in a product, projection
into some finite subproduct is monic: apply the zero-intersection
definition to the coordinate kernels. Such a monomorphism need
not split, as the preceding counterexample shows.

\begin{lemma}\label{relative-extension}
Let $M$ be quasi-injective and let $N$ be a direct summand of
$M$. Every homomorphism $L\to N$, where $L\subseteq M$,
extends to a homomorphism $M\to N$. Consequently, every
monomorphism $N\to M$ splits.
\end{lemma}
\begin{proof}
Include $N$ in $M$, extend to an endomorphism of $M$, and
compose with a projection onto $N$. For the last assertion,
apply this extension property to the inverse map on the image
of the monomorphism.
\end{proof}

A module $U$ is directly finite if $U\cong U\oplus Z$ implies
$Z=0$. A quasi-injective directly finite module has cancellation
by Warfield's theorem \cite[Corollary 2.6]{Warfield}.
In particular, this applies to a quasi-injective module of finite
uniform dimension, and to a finitely cogenerated quasi-injective
module. In the latter case one can also test direct finiteness on
the finite essential socle.

\begin{lemma}[Consequence of Warfield's cancellation theorem]
\label{socle-comparison}
Let $M$ be quasi-injective, and let $U,V$ be direct summands of $M$.
Suppose that $U$ is directly finite and that $U\oplus X$ is a
direct summand of $M$ with the indicated copy of $U$.
If there are monomorphisms
\[
U\xrightarrow f V\xrightarrow g U\oplus X,
\]
then $V\cong U\oplus Y$ for a direct summand $Y$ of $X$.
No condition on $gf$ and no finiteness condition on $V$ are required.
\end{lemma}
\begin{proof}
Lemma~\ref{relative-extension} splits both monomorphisms.
Thus, for suitable modules $V_0,Z$,
\[
V\cong U\oplus V_0,\qquad
U\oplus X\cong V\oplus Z\cong U\oplus V_0\oplus Z.
\]
The module $U$ is quasi-injective, being a summand of $M$.
Warfield's theorem therefore cancels $U$ and gives
$X\cong V_0\oplus Z$. Take in $X$ the corresponding copy $Y$ of $V_0$.
\end{proof}

For the second application below, recall that a module is
nonsingular if no nonzero element has an essential right annihilator.
Its uniform dimension is the supremum of the sizes of independent
families of nonzero submodules.

\begin{lemma}\label{nonsingular-finite-projection}
Let $M$ be nonsingular and quasi-injective.
If $N,P$ are direct summands of $M$, the kernel of every map
$N\to P$ is a direct summand of $N$.
Consequently, if $N$ has finite uniform dimension and
$M\cong\prod_{t\in T}M_t$, the projection of $N\subseteq M$
into some finite subproduct is monic.
\end{lemma}
\begin{proof}
A map with essential kernel and nonsingular target is zero:
if its kernel is $K$, then for each element $x$ of its domain,
the right ideal $\{r:xr\in K\}$ is essential and annihilates
the image of $x$.

Now let $f:N\to P$, put $K=\ker f$, and choose $L\subseteq N$
such that $K\oplus L$ is essential in $N$.
The inverse map $f(L)\to L$ extends to a map $b:P\to N$ by
Lemma~\ref{relative-extension}. Put $e=bf$.
On $K\oplus L$, the map $e$ is the projection onto $L$.
Both $e^2-e:N\to N$ and $f-fe:N\to P$ therefore vanish
on an essential submodule and are zero.
Thus $e$ is idempotent and $\ker e=\ker f$, proving the first assertion.

For a finite $F\subseteq T$, let $K_F$ be the kernel of
$N\to\prod_{t\in F}M_t$.
The target is a direct summand of $M$, so $K_F$ is a summand of $N$.
Choose $F$ for which the finite integer $\operatorname{udim}K_F$
is least. For any finite $G\supseteq F$, the summand $K_G$ of $N$
is contained in $K_F$. Restricting a retraction $N\to K_G$
to $K_F$ shows that $K_G$ is a summand of $K_F$.
Minimality and additivity of finite uniform dimension imply
$K_G=K_F$. Since $\bigcap_GK_G=0$, we obtain $K_F=0$.
\end{proof}

\Needspace{16\baselineskip}
\begin{theorem}\label{product-comparison}
Suppose a quasi-injective module $M$ has product decompositions
\[
A\times\prod_{i\in\Lambda}C_i\cong M\cong
B\times\prod_{j\in\Delta}D_j.
\]
Assume either of the following:
\begin{enumerate}
\item $A,B$, and each $D_j$ are finitely cogenerated;
\item $M$ is nonsingular, and $A,B$, and each $D_j$ have finite
uniform dimension.
\end{enumerate}
There are finite sets $\Lambda_0\subseteq\Lambda$,
$\Delta_0\subseteq\Delta$, and a direct summand
$Y\mid\prod_{i\in\Lambda_0}C_i$ such that
\begin{equation}\label{finite-comparison}
A\oplus Y\cong B\oplus\prod_{j\in\Delta_0}D_j.
\end{equation}
The module $Y$ is quasi-injective. If $M$ is injective, then $Y$ is injective.
\end{theorem}
\begin{proof}
Regard the displayed factors as summands of $M$.
Under the first hypothesis, finite cogeneration gives a finite
$\Delta_0$ such that projection into
$V=B\oplus\prod_{j\in\Delta_0}D_j$ is monic on $A$.
The module $V$ is again finitely cogenerated, so its projection
into $A\oplus X$, for a suitable finite subproduct
$X=\prod_{i\in\Lambda_0}C_i$, is monic.
Under the second hypothesis, Lemma~\ref{nonsingular-finite-projection}
gives the same two monomorphisms, since $A$ and $V$ have finite
uniform dimension. In either case $A$ is directly finite.
Apply Lemma~\ref{socle-comparison} with $U=A$ to obtain $Y$.
The last assertions follow because $Y$ is a summand of $M$.
\end{proof}

The second hypothesis includes modules with zero socle, for example
finite-dimensional rational vector spaces viewed as $\mathbb Z$-modules.
The kernel-splitting argument is the familiar regularity argument
for nonsingular quasi-injective modules; the finite comparison is
an application of this argument and classical cancellation.
The following example shows why direct finiteness alone cannot
replace the finite-projection hypotheses, even for an injective
common module.

\begin{example}\label{prufer-boundary}
Fix a prime $p$, let $C=\mathbb Z(p^\infty)$ be the Pr\"ufer $p$-group,
and put $P=C^{\mathbb N}$.
There are decompositions
\[
C\oplus P\cong P\cong\mathbb Q\oplus P,
\]
but no finite sets $F,G\subseteq\mathbb N$ and no summand
$Y\mid C^F$ satisfy
\[
C\oplus Y\cong\mathbb Q\oplus C^G.
\]
The same failure holds after interchanging $C$ and $\mathbb Q$.
All displayed modules are injective; $C$ and $\mathbb Q$ are
uniform and hence directly finite. Each family factor $C$ is
finitely cogenerated. Thus neither distinguished factor in
the first hypothesis of Theorem~\ref{product-comparison} can be
replaced merely by a directly finite module, or even by a uniform
injective module, while retaining the other hypotheses.
\end{example}
\begin{proof}
The group $P$ is divisible. Its torsion subgroup $T$ is divisible
as well: given $n\geq1$ and $t\in T$, choose $y\in P$ with
$ny=t$. If $mt=0$, then $mny=0$, so $y\in T$.
Hence $T$ is injective and splits off $P$.
The quotient $V=P/T$ is torsion-free and divisible, and therefore
is a rational vector space.

Partition $\mathbb N$ into infinite subsets $I_m$.
Choose $c_n\in C$ of order $p^n$ and let $z_m\in P$ have
$n$th coordinate $c_n$ for $n\in I_m$, and zero otherwise.
Any nonzero finite integral linear combination of the $z_m$
has coordinates of unbounded order, and is therefore not torsion.
Consequently the classes $z_m+T$ are rationally independent in $V$.
Thus $\dim_{\mathbb Q}V$ is infinite, so
$V\cong\mathbb Q\oplus V$ and $P\cong\mathbb Q\oplus P$.
Separating one coordinate gives $P\cong C\oplus P$.

Every finite power of $C$, and each of its subgroups, is torsion.
Hence $C\oplus Y$ is torsion whereas $\mathbb Q\oplus C^G$ is not.
Interchanging the two distinguished factors gives the reverse
torsion comparison. Finally $\Soc C$ is the unique subgroup of
order $p$ and is essential in $C$, so $C$ is finitely cogenerated;
$\Soc\mathbb Q=0$. Both groups are uniform divisible groups,
which proves the remaining assertions.
\end{proof}

\section{Beyond modules: Bergman's Question 3.2}\label{other-structures}
We now consider the other-structures question. The ordinary
group version for finite-support tuples is already noted in
\cite[Section 3]{Bergman}.
We give a categorical formulation, a version for pointed
algebras with a group reduct, and a semidirect-product
formulation with all normal factors and actions specified.
The finite-factorization assumptions in these results are
part of their hypotheses.

\subsection{An additive formulation and its product dual}

An additive category is \emph{idempotent-complete} if each idempotent
endomorphism splits. The following elementary lemma isolates the
categorical content of the argument.

\begin{lemma}[Nested split monomorphisms]\label{lem:sandwich}
Let $\C$ be an idempotent-complete additive category. Suppose
\[
 A\xrightarrow{a}E\xrightarrow{b}A\oplus X,
 \qquad ba=\iota_A,
\]
and $b$ has a left inverse. Then $E\cong A\oplus Y$ for some direct
summand $Y$ of $X$.
\end{lemma}
\begin{proof}
Set $r=\pi_A b$. Then $ra=\id_A$.
Splitting the idempotent $\id_E-ar$ gives $E\cong A\oplus Y$.
In this decomposition $b$ has block form
\[
 b=\begin{pmatrix}\id_A&0\\0&j\end{pmatrix}:A\oplus Y\longrightarrow A\oplus X.
\]
Indeed, $ba=\iota_A$ gives the first column, and $r$ vanishes on $Y$.
If $c$ is a left inverse of $b$, its lower-right block is a left inverse
of $j$. Thus $j:Y\to X$ is a split monomorphism. Splitting the associated
idempotent on $X$ shows that $Y$ is a direct summand of $X$.
\end{proof}

\begin{definition}
In an additive category with set-indexed coproducts, an object $K$ is
\emph{small} if every morphism $K\to\bigoplus_{t\in T}X_t$
factors through a finite subcoproduct. Equivalently, the canonical map
$\bigoplus_t\Hom(K,X_t)\to\Hom(K,\bigoplus_tX_t)$ is an isomorphism.
\end{definition}

\begin{theorem}\label{thm:additive}
Let $\C$ be an idempotent-complete additive category with set-indexed
coproducts. Suppose
\[
 A\oplus\bigoplus_{i\in I}C_i
 \cong B\oplus\bigoplus_{j\in J}D_j,
\]
where $A,B$, and all $D_j$ are small.
Then finite subsets $I_0\subseteq I$, $J_0\subseteq J$, and
$Y\mid\bigoplus_{i\in I_0}C_i$ satisfy
\[
 A\oplus Y\cong B\oplus\bigoplus_{j\in J_0}D_j.
\]
No abelian-category hypothesis is required; no exchange property
is assumed for this finite conclusion.
\end{theorem}
\begin{proof}
Use the given isomorphism to regard both decompositions as
decompositions of an object $M$.
The factorization property of $A$ gives a finite $J_0$ such that its inclusion
$u:A\to M$ factors through the split inclusion
$e:E=B\oplus\bigoplus_{j\in J_0}D_j\to M$; write $u=ea$.
Finite sums of small objects are small. Applying the factorization property to $E$
gives a finite $I_0$ and a factorization
\[
 e=fb,\qquad
 f:F=A\oplus\bigoplus_{i\in I_0}C_i\longrightarrow M
\]
with $f$ the canonical split inclusion. Since $fba=u=f\iota_A$ and $f$
is monic, $ba=\iota_A$. If $p:M\to E$ retracts $e$, then $pf$ is a left
inverse of $b$. Apply Lemma \ref{lem:sandwich} with
$X=\bigoplus_{i\in I_0}C_i$.
\end{proof}

\begin{remark}
For modules, this property is called ``dually slender''
or ``small'' in Section 3 of \cite{Bergman}; finite generation implies it.
For torsion-free modules over a commutative domain, finite-rank modules
are small in the category of torsion-free modules. To see this,
choose finitely many elements spanning after extension to the fraction
field and collect the finite supports of their images. For any other
element $x$, a nonzero scalar multiple of $x$ is a linear combination
of those chosen elements. Torsion-freeness of the target coordinates
forces the image of $x$ to vanish outside the collected supports.
That category has coproducts and split idempotents, so the theorem also
unifies the finite-rank version already noted by Bergman.
\end{remark}

\begin{corollary}[Product dual]\label{cor:dual}
Let $\C$ be an idempotent-complete additive category with set-indexed
products. Assume that $A,B$, and each $D_j$ have the following property:
every morphism from any product $\prod_t X_t$ to that object factors
through a finite subproduct. If
\[
 A\times\prod_iC_i\cong B\times\prod_jD_j,
\]
then finite $I_0,J_0$ and $Y\mid\prod_{i\in I_0}C_i$ give
$A\oplus Y\cong B\oplus\prod_{j\in J_0}D_j$.
\end{corollary}
\begin{proof}
Apply Theorem \ref{thm:additive} in $\C^{\mathrm{op}}$.
The condition in the statement is exactly smallness there.
\end{proof}

\begin{corollary}[Continuous product version]\label{cor:topology}
Let $R$ be discrete. Give the modules $A,B,C_i,D_j$ the discrete
topology and the infinite products their product topologies. If
\[
 A\times\prod_iC_i\cong B\times\prod_jD_j
\]
by an $R$-linear homeomorphism, then the finite conclusion of
Corollary \ref{cor:dual} holds, with $Y$ a topological direct summand of
the finite product. No generation or cogeneration hypotheses are needed.
\end{corollary}
\begin{proof}
Work in the additive category of Hausdorff topological $R$-modules and
continuous homomorphisms. It has products and split idempotents:
a continuous idempotent $e$ splits a module as
$\im e\times\ker e$, with inverse coordinates $e$ and $1-e$.
Every continuous map from a product to a discrete module has an open
kernel. A basic product neighborhood in that kernel restricts only
finitely many coordinates, so the entire complementary tail product
is killed. The map therefore factors continuously through the indicated
finite subproduct. All the distinguished discrete modules have the
factorization property in Corollary \ref{cor:dual}.
\end{proof}

\begin{remark}
In Theorem~\ref{counterexample}, with $a\ne b$, every algebraic
isomorphism $S^a\oplus P\cong S^b\oplus P$ fails to be a homeomorphism
for these product topologies. Otherwise Corollary \ref{cor:topology}
would contradict \eqref{finite-failure}. In particular, the retraction $P\to S$ constructed at the end of
Section~\ref{counterexample-section} is discontinuous when $S$ is discrete.
This identifies a sufficient additional hypothesis that restores finite
stabilization without confusing it with finite cogeneration.
\end{remark}

\subsection{Pointed algebras with a group reduct}

We specify a broad class in which the internal, stronger conclusion holds.
Let $\mathcal V$ be a finitary variety whose signature contains group
operations, with group identity denoted by $1$, and whose additional
fundamental operations satisfy
$\omega(1,\ldots,1)=1$. Additional constants, if present, equal $1$.
Thus each algebra has an underlying group, and the one-element algebra
is a subalgebra. For additive group structures the symbol $1$ here is
replaced by $0$.

The restricted direct product $\bigoplus_i C_i$ means the subalgebra of
$\prod_i C_i$ consisting of finite-support tuples. It is closed under all
operations because they are finitary and preserve the identity.
For nonabelian groups this notation is not their categorical coproduct.

The modular law below is the group-reduct specialization of
Crawley--J\'onsson \cite[Lemma 2.3, p.803]{CJ}. Their finitary,
zero-preserving framework also includes these finite-support
products. We use this classical law for a finite-stabilization
conclusion; we do not remove exchange hypotheses from their
stronger refinement theorems.

\begin{lemma}\label{lem:variety}
In such a variety, suppose
$M=B\times V$ and $F=A\times X$ are internal decompositions, with
$A\leq B\leq F\leq M$. Put
\[
 Y=X\cap B,\qquad Z=X\cap(A\times V).
\]
Then $B=A\times Y$ and $X=Y\times Z$ as internal decompositions
of algebras.
\end{lemma}
\begin{proof}
First observe that if $P=S\times T$ and $S\leq Q\leq P$, then
$Q=S\times(Q\cap T)$. Indeed, write a member $q$ as $st$; since
$s\in Q$, the group law gives $t=s^{-1}q\in Q$.
The resulting bijection is an algebra homomorphism, being the
restriction of the original product isomorphism to subalgebras.

Apply this observation inside $F=A\times X$ to obtain $B=A\times Y$.
Consequently
$M=B\times V=(A\times Y)\times V=Y\times(A\times V)$
as algebras. Now $Y\leq X\leq M$, so a second application gives
$X=Y\times\bigl(X\cap(A\times V)\bigr)$.
\end{proof}

\begin{theorem}\label{thm:variety}
Suppose an algebra in $\mathcal V$ has internal decompositions
\[
 M=A\times\bigoplus_{i\in I}C_i
   =B\times\bigoplus_{j\in J}D_j,
\]
where $A,B$, and every $D_j$ are finitely generated as algebras.
There are finite sets $I_0,J_0$ and a direct factor
$Y\mid\bigoplus_{i\in I_0}C_i$ such that, internally,
\[
 A\times Y=B\times\bigoplus_{j\in J_0}D_j.
\]
\end{theorem}
\begin{proof}
The supports of finitely many generators show that
$A\leq E=B\times\bigoplus_{j\in J_0}D_j$ for a finite $J_0$.
A finite product of finitely generated algebras here is finitely generated:
take their generators in individual coordinates and use the group law
to form arbitrary finite tuples. Thus $E$ is finitely generated, and
$E\leq F=A\times\bigoplus_{i\in I_0}C_i$ for a finite $I_0$.
Also $M=E\times V$ where $V=\bigoplus_{j\notin J_0}D_j$.
Apply Lemma \ref{lem:variety}, with $E$ playing its role of $B$ and
$X=\bigoplus_{i\in I_0}C_i$.
\end{proof}

\begin{remark}
Examples include groups, groups with specified endomorphisms, Lie
algebras over a fixed commutative ring, and associative or nonassociative
algebras without a specified multiplicative identity. For rings and
algebras in this assertion the additive group supplies the group reduct.
The unital-ring category requires a different formulation, since its
zero-coordinate inclusions need not be unital homomorphisms.
The theorem is a sufficient, explicitly proved answer to Question 3.2;
the group case itself was already observed in \cite{Bergman}.
\end{remark}

\subsection{A semidirect-product formulation for groups}
Here $N\rtimes H$ means that $N$ is normal and $H$ is a
specified complement. The following subgroup calculation replaces
the direct-factor sandwich. It requires normality of $A$ in $F$,
as part of $F=A\rtimes X$, but not in the whole group $M$.

\begin{lemma}[Nested semidirect decompositions]\label{lem:semidirect-sandwich}
Suppose
\[
M=E\rtimes V,\qquad F=A\rtimes X,\qquad A\leq E\leq F\leq M.
\]
Let $p:F\to X$ be projection with kernel $A$. Put
\[
Y=E\cap X,\qquad Z=p(F\cap V)=X\cap AV.
\]
Then $E=A\rtimes Y$ and $X=Y\rtimes Z$.
Here $AV$ denotes the set product.
The actions are inherited from the given decompositions.
\end{lemma}
\begin{proof}
If $A\leq E\leq A\rtimes X$, writing $e=ax$ shows that
$x=a^{-1}e\in E\cap X$. Hence $E=A\rtimes(E\cap X)$
and $p(E)=Y$.
Likewise, $E\leq F\leq E\rtimes V$ gives
$F=E\rtimes W$, where $W=F\cap V$.
The restriction $p|_W$ is injective, since $W\cap A\leq W\cap E=1$.
If $p(w)\in p(E)$, then $w\in EA=E$, and hence $w=1$.
Applying $p$ to $F=EW$ therefore gives $X=Y\rtimes p(W)$.
Here $Y=p(E)$ is normal in $X$ because $E$ is normal in $F$.
For $w=ap(w)\in W$, we have $p(w)=a^{-1}w\in X\cap AV$.
Conversely, if $z=av\in X\cap AV$, then
$v=a^{-1}z\in F\cap V$ and $z=p(v)$.
Thus $p(W)=X\cap AV$.
\end{proof}

\begin{theorem}\label{thm:semidirect}
Let a group $M$ have internal decompositions
\[
M=A\rtimes\bigoplus_{i\in I}C_i
 =B\rtimes\bigoplus_{j\in J}D_j,
\]
where $A,B\trianglelefteq M$ and $A,B,D_j$ are finitely
generated. Each displayed complement is the restricted direct
product of its factors. There are finite subsets $I_0\subseteq I$, $J_0\subseteq J$ and
subgroups $Y,Z\leq X=\bigoplus_{i\in I_0}C_i$ such that
\begin{equation}\label{eq:semidirect-conclusion}
X=Y\rtimes Z,\qquad
A\rtimes Y=B\rtimes\bigoplus_{j\in J_0}D_j.
\end{equation}
The actions are inherited from $M$.
\end{theorem}
\begin{proof}
Finite supports of generators give finite sets $J_0,I_0$ with
\[
A\leq E=B\rtimes\bigoplus_{j\in J_0}D_j
 \leq F=A\rtimes\bigoplus_{i\in I_0}C_i=A\rtimes X.
\]
Here $E$ is finitely generated, being a finite semidirect
product of finitely generated groups.
Put $V=\bigoplus_{j\notin J_0}D_j$.
The subgroup $E$ is normal in $M$, since it is the inverse
image of a direct factor of $M/B$, and $M=E\rtimes V$.
Lemma~\ref{lem:semidirect-sandwich} gives the conclusion with
$Y=E\cap X$ and $Z=X\cap AV$.
\end{proof}

\paragraph{Comparison with Question 3.2.}
Bergman already observes the ordinary group case and asks
whether semidirect versions hold. The finite-support truncations
above are the same as in his proof. Crawley--J\'onsson
\cite[Lemma 2.3]{CJ} supply the modular law for internal direct
factors. Lemma~\ref{lem:semidirect-sandwich} replaces that
calculation by $Y=E\cap X$ and $Z=X\cap AV$; normality gives
$Y\trianglelefteq X$ and preserves the inherited actions.
Theorem~\ref{thm:semidirect} answers the case with normal
factors $A,B$ and restricted direct-product complements.
It makes no claim about arbitrary iterated semidirect products.

\begin{example}\label{ex:semidirect-infinite-actions}
Let
\[
W=\{(w_1,w_2,w_3)\in\mathbb F_5^3:w_1+w_2+w_3=0\}.
\]
The group $G=S_3=N\rtimes K$, with $N=A_3\cong C_3$ and
$K\cong C_2$ generated by a transposition, acts on $W$
by permuting coordinates. Let
\[
T=\bigoplus_{n\geq1}\langle t_n\rangle,\qquad |t_n|=2,
\]
and let each $t_n$ act as $-1$ on $W$.
These scalar actions commute with the $G$-action.
For a further central group $U\cong C_2$, put
\[
M=(W\rtimes(G\times T))\times U,
\qquad A=W\times U,\qquad B=W\rtimes N.
\]
The groups $A$ and $B$ are finite and normal in $M$, with
$A\cap B=W$; neither contains the other. Moreover,
\[
M=A\rtimes(G\times T)=B\rtimes(K\times U\times T).
\]
Both complements are restricted direct products of finite groups.
Choose only $U$ from the second complement and only $G$ from
the first. Then
\[
E=B\times U=A\rtimes N,\qquad X=G=N\rtimes K.
\]
Thus $Y=N$ is not a direct factor of $X=S_3$: a direct
complement would be a normal subgroup of order two.
The action of $Y$ on $W\subseteq A$, the action of $K$ on $Y$,
and the action of every tail factor $\langle t_n\rangle$ on $W$
are nontrivial. The example concerns this finite choice;
it does not exclude a direct-factor comparison after different
finite sets have been chosen.
\end{example}

\section{Extension from the socle}\label{restriction-section}
Throughout this section the ring is arbitrary. A module $M$ is
quasi-injective if each map from a submodule of $M$ to $M$
extends to an endomorphism of $M$. For $U\subseteq M$, put
\[
\mathcal N_M(U)=\{a\in\End_R(M):a(U)=0\}.
\]
For finitely generated $U$, these subgroups form a basis of
zero neighborhoods for the finite topology. It is the
topology of pointwise convergence with $M$ discrete.
Endomorphism rings are complete and Hausdorff: the equations
for additivity and scalar compatibility define a closed
subset of the product of discrete sets $M^M$. Composition
is continuous. We use the convention of \cite[Section 1]{AW}.

\Needspace{20\baselineskip}
\begin{theorem}\label{restriction}
Let $M$ be a quasi-injective module over an arbitrary ring and
put $W=\Soc M$. Restriction defines a continuous surjective
ring homomorphism
\[
\rho_M:\End_R(M)\longrightarrow\End_R(W).
\]
For every submodule $U\subseteq M$,
\begin{equation}\label{image}
\rho_M\bigl(\mathcal N_M(U)\bigr)=\mathcal N_W(U\cap W).
\end{equation}
Hence the following conditions are equivalent:
\begin{enumerate}
\item $\rho_M$ is open for the finite topologies;
\item $\Soc U$ has finite length for every finitely generated
submodule $U\subseteq M$.
\end{enumerate}
If $W$ is essential in $M$, then
$\ker\rho_M=J(\End_R(M))$.
\end{theorem}
\begin{proof}
The socle is fully invariant, so restriction is a ring
homomorphism. Quasi-injectivity gives surjectivity, and
continuity follows by testing on finite subsets of $W$.
The inclusion from left to right in \eqref{image} is immediate.
If $g\in\End_R(W)$ vanishes on $U\cap W$, the rule
\[
U+W\longrightarrow M,\qquad u+w\longmapsto g(w)
\]
is well-defined and $R$-linear. Extend it to an endomorphism
of $M$. This extension vanishes on $U$ and restricts to $g$,
proving \eqref{image}.

For a submodule $L$ of the semisimple module $W$, we claim
that $\mathcal N_W(L)$ is open if and only if $L$ has finite
length. The forward implication requires verification.
If $\mathcal N_W(V)\subseteq\mathcal N_W(L)$ for a finitely
generated $V\subseteq W$ and $L\not\subseteq V$, choose a
simple quotient of $(L+V)/V$ on which the image of $L$ is
nonzero. Semisimplicity extends the corresponding map to
$W/V$, and its simple target embeds in $W$. The resulting
endomorphism of $W$ kills $V$ but not $L$, a contradiction.
Thus $L\subseteq V$ and has finite length. The converse follows
by choosing finite generators of $L$. Apply this claim to
$L=U\cap W=\Soc U$ in \eqref{image}.

Suppose finally that $W$ is essential. If $n\in\ker\rho_M$
and $b\in\End_R(M)$, the map $1-bn$ fixes $W$ pointwise.
Its kernel is zero by essentiality. By
Lemma~\ref{relative-extension}, this monomorphism splits.
Its image contains the essential submodule $W$, so its
complement is zero. Thus $1-bn$ is invertible for every $b$,
and $n\in J(\End_R(M))$. Conversely, the surjective map
$\rho_M$ sends the radical into the radical of
$\End_R(W)$, which is zero because this endomorphism ring
is von Neumann regular. This proves the equality.
\end{proof}

\begin{remark}\label{nilpotent-kernel}
If $R/J$ is semisimple artinian and $J^d=0$, then the kernel
in Theorem~\ref{restriction} has $d$th power zero. Indeed, put
$W_k=\{m\in M:mJ^k=0\}$. Then $W_1=W$, $W_d=M$, and
every endomorphism killing $W$ sends $W_k$ into $W_{k-1}$.
This argument uses no self-injectivity assumption on $R$.
\end{remark}

Essentiality cannot be omitted from the last assertion of the
theorem. For example, $\mathbb Q$ as a $\mathbb Z$-module has
zero socle. Restriction then has the whole endomorphism field
as kernel. The openness criterion itself still applies.

\section{Semiprimary self-injective rings}\label{structure-section}
From now until the end of Section~\ref{homogeneous-section},
assume \eqref{F}. Put $S=\Soc R$. We collect the structural
facts needed in the proofs.
The arguments fix the sides of annihilators and tensor products.
A module $X$ is FP-injective if $\Ext_R^1(F,X)=0$ for every
finitely presented right module $F$. For a subset $L\subseteq R$, let
$\la(L)=\{r:rL=0\}$ and $\ra(L)=\{r:Lr=0\}$.

\begin{lemma}\label{structure}
Under \eqref{F}, every module has essential socle, and $R_R$
is an injective cogenerator. Its socle $S$ is of finite length on both
sides, contains every simple type on both sides, and satisfies
\[
S=\la(J)=\ra(J),\qquad JS=SJ=0,\qquad
\la(S)=\ra(S)=J.
\]
For every projective right module $Q$,
\[
QS=\Soc Q,
\qquad Q\otimes_RS\xrightarrow{\ \sim\ }QS.
\]
The injective hull of each simple right module is finitely
generated and projective. Every projective right module is
FP-injective, and every monomorphism between projective right
modules splits.
\end{lemma}
\begin{proof}
Nilpotence of $J$ gives a nonzero semisimple submodule in
every nonzero cyclic module. Decompose $R_R$ into finitely
many primitive projectives. Each is injective and has simple
essential socle: the hull of a simple submodule embeds and
splits, and indecomposability excludes a complement.
Distinct projective types have distinct socle types, since
injective hulls are unique. The number of projective types
equals the number of simple types, so every simple right
module occurs in $S$. This also proves the assertion about
its hull, and $R_R$ is a cogenerator.

Certainly $SJ=0$. If left multiplication by some $j\in J$
is nonzero on $\Soc(eR)$, it is monic on $eR$ by essentiality
and simplicity. It splits because $eR$ is injective. A left
inverse would give $cje=e$, which is impossible modulo $J$.
Thus $JS=0$.

Every endomorphism of $S_D$ extends to an endomorphism of
$R_R$, hence to left multiplication. The kernel of this
action is $J$: a larger ideal modulo $J$ would contain a
nonzero idempotent, whose lift $e$ would satisfy $eS=0$,
contrary to $\Soc(eR)\ne0$. Therefore
$D\cong\End_D(S_D)$. Since $S_D$ is a finite faithful
generator, this gives a Morita self-equivalence of $D$.
In particular, ${}_DS$ is finite and contains all simple
left types. The cogenerator property gives $\ra\la(L)=L$ for every
right ideal $L$: a nonzero element of $R/L$ is separated by
a homomorphism to $R$, and such a map is induced by left
multiplication by an element of $\la(L)$.
Since $S=\la(J)$, this gives $\ra(S)=J$. Applying the same
identity to $S$, and using $\la(S)=J$ proved above, gives
$\ra(J)=S$.

Semiprimary rings are perfect on both sides. Their projectives
are direct sums of finitely generated indecomposable
projectives, and their flat modules are projective
\cite{Bass}. Finite projectives are injective under
\eqref{F}. Direct sums of these are FP-injective: a map
from a finitely generated relation module has image in a
finite injective subsum and extends there. Let $u:P\to Q$ be a monomorphism between projectives.
Its restriction to every finite injective subsum of $P$
splits in $Q$. Taking their directed union shows that $u$
is pure. The tensor sequence then gives
$\Tor_1^R(Q/u(P),N)=0$ for every left module $N$, since
$Q$ is flat. The quotient is flat and hence projective,
so $u$ splits. Finally,
$\Soc(R^{(T)})=S^{(T)}=R^{(T)}S$ and the tensor assertion
hold for free modules; applying a summand projection gives
them for projectives.
\end{proof}

\begin{lemma}\label{cover-detector-lemma}
Let $0\to K_X\to Q_X\to X\to0$ be a projective cover under
\eqref{F}. There is an isomorphism of right modules
\begin{equation}\label{cover-detector}
\Tor_1^R(X,R/S)\cong\Soc K_X.
\end{equation}
The right $R$-action on the Tor group comes from the second
variable $R/S$, viewed as an $R$-$R$-bimodule.
In particular, this Tor group vanishes if and only if $X$
is projective. The left module $R/S$ is finitely presented.
\end{lemma}
\begin{proof}
Minimality gives $K_X\subseteq Q_XJ$, and hence $K_XS=0$.
Tensoring the covering sequence with $R/S$ identifies the
Tor group with
\[
\ker(K_X/K_XS\longrightarrow Q_X/Q_XS)
 =K_X\cap\Soc Q_X=\Soc K_X.
\]
Socles are essential, so vanishing is equivalent to $K_X=0$.
The last assertion follows because ${}_RS$ is finitely
generated by Lemma~\ref{structure}.
\end{proof}

\subsection{The countable free module and its hull}
Set $E=R^{(\N)}$ and $P=R^{\N}$, retaining the canonical
inclusion. An inclusion is pure when tensoring with any left
module preserves its injectivity. We first recall why this countable choice detects
the QF conclusion.

\begin{lemma}\label{countable}
Under \eqref{F}, the following are equivalent:
$R$ is QF; $E$ is injective; $P$ is projective;
the inclusion $E\subseteq P$ splits.
\end{lemma}
\begin{proof}
If $E$ is injective and $J/J^2$ has infinite length as a
right $D$-module, choose countably many distinct simple
summands of $J/J^2$ and map them nontrivially into distinct
coordinates of $S^{(\N)}$, sending a complement to zero.
The resulting map $J\to E$ has infinitely many nonzero
coordinate maps, so it cannot extend to $R\to E$. This
contradicts injectivity. Thus $J/J^2$ has finite length. Choose lifts $u_1,\ldots,u_t$
of its generators and put $K=\sum_i u_iR$. Then $J=K+J^2$;
successively substituting this equality gives $J=K+J^n$
for every $n\geq2$, so $J=K$ by nilpotence.
Every $J^n/J^{n+1}$ is generated over $D$ by the finitely
many products $u_{i_1}\cdots u_{i_n}$. Each layer therefore
has finite length, and the finite radical filtration makes
$R_R$ artinian. Right self-injectivity gives QF.

Conversely, over a QF ring direct sums of injectives are
injective, and products of projectives are projective
(coherence gives flatness by \cite{Chase}, and perfectness
gives projectivity).
This proves the required implications from QF.
The module $P$ is always injective. Therefore a splitting
of $E\subseteq P$ makes $E$ injective. If $P$ is projective,
the same inclusion splits by Lemma~\ref{structure}.
\end{proof}

\begin{proposition}\label{hull}
The inclusion $E\subseteq P$ admits a decomposition
\begin{equation}\label{hull-decomp}
P=I\oplus T,\qquad E\subseteq I,
\end{equation}
where $I$ is an injective hull of the original $E$. The map
$E\to I$ is pure and is its pure-injective envelope. Moreover,
\[
\Soc I=ES=S^{(\N)},\qquad
\Soc T\cong S^{\N}/S^{(\N)}.
\]
For $B=I/E$, the extension of $P/E=B\oplus T$ by $E$ is
the pullback along $B\oplus T\to B$ of
\begin{equation}\label{hull-extension}
\eta_I=[0\to E\to I\to B\to0].
\end{equation}
\end{proposition}
\begin{proof}
Extend the fixed map $E\to P$ across an abstract injective
hull $E\subseteq I_0$. Essentiality makes $I_0\to P$ monic,
and injectivity of its source splits it. Identify its image
with $I$. Essentiality gives $\Soc I=\Soc E$, and the
socle formulas follow from Lemma~\ref{structure} and
$\Soc P=S^{\N}$.

Since $E$ is FP-injective, $E\subseteq I$ is pure. If an
endomorphism of $I$ fixes $E$, it is monic by essentiality,
splits by injectivity, and is onto since its image contains
$E$. Thus $E\to I$ is left minimal. Every map from $E$ to
a pure-injective module extends across this pure inclusion;
hence this is a pure-injective envelope. This is the
minimality appearing in \cite{KS}; the sum--hull--product
construction itself is classical \cite{Levine}.

Quotienting \eqref{hull-decomp} by $E$ gives $B\oplus T$.
The pullback of $I\to B$ along the projection has middle
term $I\times_B(B\oplus T)\cong I\oplus T=P$, proving
the last assertion.
\end{proof}

\section{The countable hull quotient}\label{quotient-section}
We retain $E\subseteq I\subseteq P$ from Proposition~\ref{hull}
and put $B=I/E$. All maps below refer to this fixed inclusion.

\subsection{The hull extension and homological detection}
\begin{theorem}\label{quotient}
Let $B=I/E$, with $I$ as in Proposition~\ref{hull}, and let
$K_I$ be the kernel of a projective cover of $I$. Then
$BS=0$ and $BJ^{d-1}=0$. Moreover,
\begin{align}
\Ext_R^1(D,E)&\cong\Hom_R(D,B)\cong\Soc B,\label{ext-b}\\
\Soc K_I&\cong\Tor_1^R(I,R/S)
 \cong B\otimes_RS\cong(B/BJ)\otimes_DS.\label{tor-b}
\end{align}
The right $D$-action in \eqref{ext-b} is induced by
precomposition with left multiplication on $D$. In
\eqref{tor-b}, the right $R$-actions come from the bimodules
$S$ and $R/S$.
The quotient map $p:I\to B$ induces natural isomorphisms
\begin{equation}\label{natural-tor}
\Tor_n^R(p,N):\Tor_n^R(I,N)\xrightarrow{\ \sim\ }
\Tor_n^R(B,N)\qquad(n\geq1)
\end{equation}
for every left $R$-module $N$.
Each of $B=0$ and projectivity of $I$ is equivalent to $R$
being QF.
\end{theorem}
\begin{proof}
Since $IS\subseteq\Soc I=ES$, we have $BS=0$.
Also $J^{d-1}\subseteq S$, so $BJ^{d-1}=0$; for $d=1$
this says $B=0$.
Every map $D\to I$ has image in $\Soc I\subseteq E$.
Applying $\Hom_R(D,-)$ to the hull extension therefore
identifies $\Hom_R(D,B)$ with $\Ext_R^1(D,E)$.
Evaluation at $1+J$ gives the second isomorphism in
\eqref{ext-b}, with the stated action.

The module $E$ is projective and $E\subseteq I$ is pure.
In the Tor sequence, all positive Tor groups of $E$ vanish,
and $E\otimes_RN\to I\otimes_RN$ is injective. This proves
\eqref{natural-tor}, including degree one. Apply it to the
bimodule $N=R/S$. Tensoring $0\to S\to R\to R/S\to0$
with $B$, the multiplication map $B\otimes_RS\to B$ is zero
because $BS=0$. Hence
$\Tor_1^R(B,R/S)\cong B\otimes_RS$.
Lemma~\ref{cover-detector-lemma} now gives \eqref{tor-b}.
Since $JS=0$, the last tensor product factors through
$B/BJ$ and $D$.

If $I$ is projective, the inclusion $E\subseteq I$ splits
by Lemma~\ref{structure}, and essentiality forces $I=E$.
The converse is immediate. Lemma~\ref{countable} identifies
$I=E$ with the QF conclusion.
\end{proof}

The groups in \eqref{ext-b} and \eqref{tor-b} measure different
parts of $B$: its socle and its top tensored with $S$.
They are not asserted to be isomorphic. Both vanish exactly
when $B=0$, since a nonzero module over a semiprimary ring
has nonzero top and socle, and ${}_DS$ is a faithful generator.
The Tor obstruction has the following stronger persistence.

\begin{proposition}\label{finite-blocks}
For each positive integer $m$, the short exact sequence
$\eta_I$ is isomorphic to its $m$-fold direct sum, with
the end terms transported by the isomorphisms. In particular,
$B\cong B^{\oplus m}$.
For a finite $F\subseteq\N$, set
\[
E_{\bar F}=R^{(\N\setminus F)},\qquad
I_{\bar F}=\ker(\pi_F|_I),
\]
where $\pi_F:R^{\N}\to R^F$ is coordinate projection. Then
\[
I=R^F\oplus I_{\bar F},\qquad
E_{\bar F}\subseteq_{\rm ess}I_{\bar F},\qquad
I_{\bar F}/E_{\bar F}\cong B.
\]
If $B\ne0$, every simple type occurring in either $\Soc B$
or $B/BJ$ has infinite multiplicity.
\end{proposition}
\begin{proof}
Choose an isomorphism $E^{\oplus m}\to E$. Since $m$ is
finite, $I^{\oplus m}$ is an injective hull of $E^{\oplus m}$.
Uniqueness of injective hulls extends the chosen isomorphism
to $I^{\oplus m}\to I$, and passage to quotients gives the
isomorphism of short exact sequences.

The finite truncation of any $x\in I$ lies in $R^F\subseteq E$,
so subtracting it gives $I=R^F\oplus I_{\bar F}$.
Essentiality of $E$ in $I$ gives essentiality of $E_{\bar F}$
in $I_{\bar F}$; inclusion induces the quotient isomorphism.
Finally, taking $m=2$ shows that the socle and top of $B$
are each isomorphic to their doubles. A nonzero multiplicity
$\kappa=2\kappa$ is infinite.
\end{proof}

\begin{theorem}\label{persistence}
Assume \eqref{F}, and let $X$ be a right module with $XS=0$.
Choose projective covers
\[
q_n:Q_n\twoheadrightarrow X_n,
\qquad X_0=X,\quad X_{n+1}=\ker q_n\quad(n\geq0).
\]
Then
\begin{equation}\label{syzygy-socle}
X_nS=0,\qquad
\Soc X_{n+1}=\Soc Q_n
 \cong(X_n/X_nJ)\otimes_DS.
\end{equation}
If $X\ne0$, every $X_{n+1}$ is an essential nonzero submodule
of $Q_n$, and
\begin{equation}\label{all-tor}
\Tor_n^R(X,R/S)\cong\Soc X_n
 =\Soc Q_{n-1}\ne0\qquad(n\geq1).
\end{equation}
Consequently, if $B\ne0$, then
\[
\operatorname{pd}_R B=\operatorname{pd}_R I=\infty,
\qquad
\ell_R\Tor_n^R(I,R/S)
 =\ell_R\Tor_n^R(B,R/S)=\infty
\quad(n\geq1).
\]
\end{theorem}
\begin{proof}
For $n\geq1$, minimality gives $X_n\subseteq Q_{n-1}J$,
so $X_nS=0$ since $JS=0$. The case $n=0$ is the hypothesis.
Thus $q_n(Q_nS)=0$, and Lemma~\ref{structure} gives
\[
\Soc Q_n=Q_nS\subseteq X_{n+1}.
\]
The socle of a submodule is its intersection with the socle
of the ambient module. This proves the equality in
\eqref{syzygy-socle}. Multiplication and the isomorphism on
tops induced by a cover give
\[
\Soc Q_n\cong Q_n\otimes_RS
 \cong(Q_n/Q_nJ)\otimes_DS
 \cong(X_n/X_nJ)\otimes_DS.
\]
If $X_n\ne0$, then $Q_n\ne0$ and its essential socle is
nonzero. Thus $X_{n+1}$ is nonzero and essential in $Q_n$.
Induction proves this for every $n$. Dimension shifting and
Lemma~\ref{cover-detector-lemma} give \eqref{all-tor}.
Apply this to $X=B$ and use \eqref{natural-tor}.
By Proposition~\ref{finite-blocks}, $B\cong B\oplus B$.
Additivity of Tor implies that each of these nonzero
semisimple Tor modules is isomorphic to its double.
Every simple type occurring in it therefore has infinite
multiplicity, proving the infinite-length assertion.
\end{proof}

\begin{corollary}\label{finite-tor-test}
Under \eqref{F}, for any fixed $n\geq1$, the following are equivalent:
\[
R\text{ is QF};\qquad
\ell_R\Tor_n^R(I,R/S)<\infty;\qquad
\ell_R\Ext_R^1(D,E)<\infty.
\]
\end{corollary}
\begin{proof}
Use Theorem~\ref{persistence}, \eqref{ext-b}, and
Proposition~\ref{finite-blocks}.
\end{proof}

Short radical length of $B$ is therefore compatible with a
nonterminating projective resolution. In particular, the
bound $BJ^{d-1}=0$ supplies no homological contradiction.

\subsection{A presentation of the extension and socle descent}
We first identify the map representing the fixed extension
in a projective presentation.

\begin{proposition}\label{graph}
Lift $E\subseteq I$ to a projective cover $Q_I\to I$ and
split the lift, writing $Q_I=E\oplus Q_0$. There is a
projective cover $Q_0\to B$ with kernel $L$ and a map
$\tau:L\to E$ such that
\[
K_I=\{(-\tau(l),l):l\in L\},\qquad
\partial(\tau)=\eta_I.
\]
The map $\tau$ extends to $Q_0\to E$ if and only if $B=0$.
\end{proposition}
\begin{proof}
Projectivity of $E$ gives the lift, which is monic and
splits by Lemma~\ref{structure}. Write the covering map as
$(e,q)\mapsto e+u(q)$. The induced epimorphism
$Q_0\to B$ has kernel $L=\{q:u(q)\in E\}$; set
$\tau(l)=u(l)$. This gives the displayed graph.
Since $K_I\subseteq(E\oplus Q_0)J$, its projection gives
$L\subseteq Q_0J$. Since $J$ is nilpotent, $Q_0J$ is
superfluous in $Q_0$, so $Q_0\to B$ is a projective cover.
The quotient by the graph is the pushout of
$0\to L\to Q_0\to B\to0$ along $\tau$, proving the
boundary formula with this sign convention.

An extension $t:Q_0\to E$ of $\tau$ induces a retraction
$I\to E$ by $[(e,q)]\mapsto e+t(q)$. Conversely a
retraction composed with $u$ extends $\tau$.
Essentiality makes such a retraction equivalent to $I=E$.
\end{proof}

\begin{theorem}\label{descent}
Put $A=R/S$ and
\[
I'=\ann_I(S),\qquad E'=\ann_E(S)=EJ=J^{(\N)}.
\]
Then
\begin{equation}\label{descent-identities}
I=E+I',\qquad E\cap I'=E',\qquad B\cong I'/E'.
\end{equation}
The module $J_A$ is a faithful, finitely cogenerated injective
cogenerator, and left multiplication induces an isomorphism
$A\xrightarrow{\sim}\End_A(J)$.
The right $A$-module $I'$ is injective and is
an injective hull of $E'$, and $E'\subseteq I'$ is pure
over $A$. After inflation to $R$, the original extension
$0\to E\to I\to B\to0$ is the pushout of
$0\to E'\to I'\to B\to0$ along $E'\hookrightarrow E$.
If $d\geq2$, the radical of $A$ has $(d-1)$st power zero.
\end{theorem}
\begin{proof}
For $x\in I$, the map $S\to I$, $s\mapsto xs$, lands in
$\Soc I=S^{(\N)}\subseteq E$. Since $S_R$ is finitely
generated, its image lies in a finite subsum $E_F$ of $E$.
This subsum is injective, so the map extends to $R\to E_F$.
If $e$ is the image of $1$, then $(x-e)S=0$. Hence
$x\in E+I'$. Also
\[
E\cap I'=\ann_E(S)=J^{(\N)}=EJ
\]
by $\la(S)=J$. This proves \eqref{descent-identities}.

The functor $\Hom_R(A,-)$ is right adjoint to exact
inflation from $A$-modules to $R$-modules, and preserves
injectives. Under evaluation at $1+S$, it identifies
$\Hom_R(A,R)$ with $J$ and $\Hom_R(A,I)$ with $I'$.
Thus both are injective over $A$.
The socle of $J_A$ is $J\cap S$, which is finite and essential.
Every simple $A$-module embeds in $R_R$ and, being annihilated
by $S$, its image lies in $\la(S)=J$.
Thus $J_A$ is an injective cogenerator.
Its right annihilator is $\ra(J)=S$, so it is faithful over $A$.
Every $A$-endomorphism of $J$ extends, as an $R$-map $J\to R$,
to left multiplication by some $a\in R$.
The kernel of $R\to\End_A(J)$ is $\la(J)=S$, proving the
asserted ring isomorphism.

Every nonzero submodule
of $I'$ meets $E$, and its intersection with $E$ lies in
$E'$. Hence $E'$ is essential in $I'$. The direct sum $E'$
of copies of the injective $A$-module $J$ is FP-injective;
therefore its inclusion in $I'$ is pure over $A$.

The map $E\oplus I'\to I$, $(e,y)\mapsto e+y$, is
surjective and has kernel
$\{(z,-z):z\in E'\}$. This identifies $I$ with the required
pushout, preserving the quotient map to $B$.
Finally, $\Jac A=(J+S)/S$ and $J^{d-1}\subseteq S$, so
$(\Jac A)^{d-1}=0$ when $d\geq2$.
\end{proof}

The descent keeps the quotient and the extension, and also
the faithful injective cogenerator and its displayed endomorphism
ring. It changes the coefficient module from $R$ to $J$. Neither
self-injectivity of $A$ nor projectivity of $J_A$ follows
from the construction. Thus Theorem~\ref{descent} is not an
induction proof of Faith's conjecture. In particular, the
all-degree comparison \eqref{natural-tor} over $R$ cannot
simply be transferred to $A$, since $E'$ has not been shown
projective over $A$.
Indeed, $J_A^{(\N)}$ is injective exactly when its essential
injective hull $I'$ equals it, which by \eqref{descent-identities}
is equivalent to $B=0$.
These conditions are also equivalent to $\Sigma$-injectivity
of $J_A$: one implication is immediate, and the other uses that
when $R$ is QF, $A$ is right artinian and direct sums of injective
$A$-modules are injective.
This reformulation retains the same unproved finiteness issue.

\Needspace{12\baselineskip}
\begin{example}\label{descent-boundary}
The two losses in the descent occur even when $R$ is QF.
Let $k$ be any field and take
\[
R=k[x,y]/(x^2,y^2),\qquad S=kxy,\qquad A=R/S.
\]
The coefficient of $xy$ in a product gives a nondegenerate
associative bilinear form on $R$, so $R$ is Frobenius.
The ring $A$ has basis $1,\bar x,\bar y$ and square-zero
radical. The map $k\bar x\to A$ sending $\bar x$ to
$\bar y$ cannot extend to $A\to A$, since left
multiplication sends $\bar x$ into $k\bar x$. Thus $A$
is not self-injective. The injective $A$-module $J=J(R)$
has length three and socle $kxy$, whereas $A$ has length
three and socle $k\bar x\oplus k\bar y$. If $J_A$ were
projective, it would be finite free over the local ring
$A$; length would force rank one, contradicting the socle
calculation. Thus the descended coefficient need not be
projective either.
\end{example}

\section{Socles of cyclic modules and Faith's conjecture}\label{homogeneous-section}
The finite-socle condition in Theorem~\ref{restriction} can
be tested on cyclic submodules when the ambient module is
injective and the ring satisfies \eqref{F}.

\subsection{A projectivity criterion}
\begin{theorem}\label{projectivity}
Assume \eqref{F}, and let $M$ be an injective right module.
The following conditions are equivalent:
\begin{enumerate}
\item $\rho_M:\End_R(M)\to\End_D(\Soc M)$ is open;
\item $\Soc U$ has finite length for every finitely generated
submodule $U\subseteq M$;
\item $\Soc(xR)$ has finite length for every $x\in M$;
\item every finite subset of $M$ is contained in a finitely
generated projective-injective direct summand of $M$;
\item $M$ is projective.
\end{enumerate}
\end{theorem}
\begin{proof}
Theorem~\ref{restriction} gives (i)$\Leftrightarrow$(ii),
and (ii)$\Rightarrow$(iii) is immediate.
Suppose (iii), and take $x_1,\ldots,x_m\in M$.
The hull of $x_1R$ is the hull of its finite essential
socle, hence a finite direct sum of hulls of simples.
By Lemma~\ref{structure} it is finitely generated,
projective and injective. Extending $x_1R\hookrightarrow M$
across this hull gives a monomorphism, since $x_1R$ is
essential in its hull. Its image $P_1$ splits in $M$.

Write $M=P_1\oplus M_1$, and project $x_2$ into $M_1$.
Its cyclic submodule again has finite socle by (iii).
As $M_1$ is injective, the same argument embeds its finite
projective-injective hull as a summand of $M_1$.
Continuing with the projections of $x_3,\ldots,x_m$ gives
a finite sum of such summands containing the original
$x_1,\ldots,x_m$. This proves (iv).

Under (iv), the finite projective submodules of $M$ form a
directed covering family: generators of two of them are
contained in a third. Their union is $M$, so $M$ is flat,
and perfectness makes it projective. Conversely, a finite
subset of a projective lies in a finite sum of finitely
generated indecomposable projectives. This finite sum has
finite socle, and so does each of its submodules.
Thus (v)$\Rightarrow$(ii), completing the proof.
\end{proof}

The passage from (iii) to (iv) uses injectivity of the
successive complements. It does not assert that finite
socles are preserved by arbitrary quotients.
Injectivity cannot be weakened to quasi-injectivity here:
over $k[t]/(t^2)$, the simple module is quasi-injective and
its socle restriction is the identity, but it is not
projective.

\subsection{A cyclic module of Loewy length two}
The preceding criterion reduces the remaining finite-tuple
condition in the fixed hull $I$ to a single element. A
failure of the QF conclusion has a particularly uniform
cyclic witness.

\begin{theorem}\label{homogeneous}\label{two-layer-witness}
Assume \eqref{F}, and retain the fixed inclusion
$E=R^{(\mathbb N)}\subseteq I\subseteq R^{\mathbb N}$.
The following are equivalent:
\begin{enumerate}
\item $R$ is not QF;
\item there are a primitive idempotent $e\in R$, simple right
modules $Q=eR/eJ$ and $T$, and a cyclic submodule $U\subseteq I$
such that
\begin{equation}\label{homogeneous-witness}
U/UJ\cong Q,\quad UJ=\Soc U=U\cap E\cong T^{(\N)},\quad UJ^2=0.
\end{equation}
Moreover, $U$ can be chosen essential in an injective direct
summand $H$ of $I$ with $\Soc H=\Soc U$.
\end{enumerate}
When these equivalent conditions hold, the construction below
gives a nonzero simple submodule $(E+U)/E\cong Q$ of $B=I/E$
whose pullback extension is represented by a surjection
\[
f:eJ\longrightarrow T^{(\mathbb N)}\subseteq E,
\qquad f(eJ^2)=0,
\]
which has no extension $eR\to E$.
For $\Gamma=\End_R(H)$,
\[
\ker\rho_H=\Jac\Gamma,\quad (\Jac\Gamma)^d=0,\quad
\Gamma/\Jac\Gamma\cong\End_R(T^{(\N)}).
\]
The quotient topology is discrete, whereas the target finite
topology is nondiscrete.
\end{theorem}
\begin{proof}
The radical-layer argument in the proof of
Lemma~\ref{countable} shows that finite length of $J/J^2$
would make $R$ right artinian and hence QF.
Consequently, under (i), $J/J^2$ has infinite length.
Decompose $R_R$ into finitely many primitive projectives.
For one of them, $eJ/eJ^2$ has infinite length.
Since $D=R/J$ has only finitely many simple types, one simple
module $T$ occurs there with infinite multiplicity.
Choose a surjection
\[
eJ/eJ^2\twoheadrightarrow L\cong T^{(\mathbb N)}
\]
and embed $L$ in $\Soc E$, using distinct coordinates.
Write $f:eJ\to L\subseteq I$ for the resulting map.

Injectivity of $I$ extends $f$ to $h:eR\to I$.
Put $U=h(eR)$. The kernel of $h$ is a proper submodule of
the local module $eR$, so it is contained in $eJ$.
Hence
\[
U/UJ\cong eR/eJ=Q,\qquad UJ=h(eJ)=L,
\qquad UJ^2=f(eJ^2)=0.
\]
In particular, $L$ is a nonzero semisimple submodule of $U$.
If $L$ were a proper submodule of $\Soc U$, choose a simple
submodule $N\subseteq\Soc U$ with $N\cap L=0$.
Its image in the simple module $U/L$ is nonzero, so
$U=L\oplus N$; then $UJ=0$, contradicting $UJ=L\ne0$.
Thus $\Soc U=L$, which is essential in $U$ by semiprimaryness.

We have $L\subseteq U\cap E$. If this inclusion were proper,
simplicity of $U/L$ would give $U\subseteq E$.
But a cyclic submodule of $E$ lies in a finite free subsum,
whose socle has finite length. This contradicts
$\Soc U\cong T^{(\mathbb N)}$.
Therefore $U\cap E=L$.
An injective hull of $U$ embeds in $I$ by essentiality and
injectivity of $I$; its image $H$ splits because its source
is injective. Essentiality gives $\Soc H=\Soc U$.
This proves (ii). Conversely, over a QF ring every cyclic
module has finite length, excluding (ii).

Any map $eR\to E$ has image in a finite coordinate subsum,
so its restriction cannot equal the surjection $f$ onto $L$.
Thus $\partial(f)\in\Ext_R^1(Q,E)$ is nonzero.
The map
\[
E\oplus eR\longrightarrow E+U,\qquad (a,r)\longmapsto a+h(r),
\]
has kernel $\{(-f(r),r):r\in eJ\}$:
if $h(r)\in E$, then $h(r)\in U\cap E=L$, and hence
$r\in eJ$ because $h^{-1}(L)=eJ$.
Thus $0\to E\to E+U\to Q\to0$ is the pushout representing
$\partial(f)$, and also the pullback of $0\to E\to I\to B\to0$
along the inclusion $(E+U)/E\hookrightarrow B$.
The radical assertions follow from Theorem~\ref{restriction}
and Remark~\ref{nilpotent-kernel}. Since $U$ is cyclic,
$\mathcal N_H(U)$ is open, and the neighborhood image formula
gives
\[
\rho_H(\mathcal N_H(U))
 =\mathcal N_{\Soc H}(\Soc U)=0.
\]
Consequently $\ker\rho_H$ is open and the quotient topology
is discrete. A finitely generated submodule of $T^{(\N)}$
has finite length and has a nonzero semisimple complement.
Projection onto that complement is a nonzero endomorphism
annihilating the given submodule. Hence zero is not open
in the target finite topology.
\end{proof}

The simple top $Q$ and the socle type $T$ need not be isomorphic.

\begin{proposition}\label{socle-ext-image}
Let $e$ be primitive, $Q=eR/eJ$, and $V_e=eJ/eJ^2$.
Inclusion $i:S^{(\N)}\hookrightarrow E$ gives an isomorphism
of abelian groups
\begin{equation}\label{socle-ext-quotient}
\operatorname{Im}\Ext_R^1(Q,i)\cong
\frac{\Hom_D(V_e,S^{(\N)})}
      {\Hom_D(V_e,S)^{(\N)}}.
\end{equation}
The denominator consists of the maps with finitely many
nonzero coordinate maps. If $V_e$ has infinite length,
the displayed image is nonzero.
More precisely, for a fixed quotient $q:V_e\twoheadrightarrow
T^{(\N)}$, an embedding $\tau:T\hookrightarrow S$, and
$\Delta=\End_R(T)$, there is an additive injection
\begin{equation}\label{reduced-product-ext}
\Delta^{\N}/\Delta^{(\N)}\lhook\joinrel\longrightarrow
\operatorname{Im}\Ext_R^1(Q,i),\qquad
[(\lambda_n)]\longmapsto[(\tau\lambda_nq_n)_n].
\end{equation}
\end{proposition}
\begin{proof}
Since $eR$ is projective and every map $eR\to S^{(\N)}$
kills $eJ$, its presentation of $Q$ identifies
\[
\Ext_R^1(Q,S^{(\N)})=\Hom_D(V_e,S^{(\N)}).
\]
The kernel of $\Ext_R^1(Q,i)$ consists of maps whose inclusion
in $E$ extends to $eR$. Such an extension has only finitely
many nonzero coordinates, because $eR$ is cyclic.
Conversely, every finite family of maps $eJ\to S$ extends
coordinatewise to $eR\to R$ by right self-injectivity.
The finitely many extensions combine to a map $eR\to E$.
This proves the kernel description and \eqref{socle-ext-quotient}.
If $V_e$ has infinite length, a quotient onto $T^{(\N)}$
followed by the coordinate embeddings into $S^{(\N)}$
has infinitely many nonzero coordinate maps, and gives a
nonzero class in the quotient.
In the last formula each $q_n$ is onto, and $\tau$ is monic,
so its $n$th coordinate is nonzero exactly when $\lambda_n\ne0$.
The kernel description therefore proves
\eqref{reduced-product-ext}. Only its additive structure is
asserted; no action of $\Delta$ on all of $S$ is needed.
\end{proof}

\paragraph{Comparison with earlier constructions.}
The countable sum--hull--product construction and the
$\Sigma$-injectivity criterion are classical
\cite[Corollary 7]{HZ}; see also Levine \cite{Levine} on
hulls of semisimple modules. Nicholson--Yousif
\cite[proof of Proposition 4]{NY97} give the second-socle
pairing, and their Theorem 2 gives a QF criterion from left
countable generation of the whole second socle. Iovanov
\cite{Iov} uses Loewy-layer cardinalities, while
Ara--Nicholson--Yousif \cite{ANY} require compatible bimodule
pairings in their local construction. Here the retained data
are the fixed $E\subseteq I$, the primitive top $Q$, the
coordinate quotient $T^{(\N)}$, and its Baer class.
Proposition~\ref{socle-ext-image} computes the kernel of the
induced Ext map, and Proposition~\ref{opposite-second-layer-test}
uses the same quotient for the left extension obstruction.
Neither left countable generation of the whole second socle
nor compatibility of the two actions is assumed to follow
from \eqref{F}. Section~\ref{nonopen-section} gives an explicit
application with a different injective coefficient.

\begin{corollary}\label{second-layer-test}
Under \eqref{F}, $R$ is QF if and only if every cyclic
submodule $U\subseteq I$ with $UJ^2=0$ and simple top
has finite socle.
\end{corollary}
\begin{proof}
The forward implication follows from artinianity; the
reverse implication follows from Theorem~\ref{two-layer-witness}.
\end{proof}

\begin{remark}\label{local-finite-extension-gap}
The construction specifies a family of maps, not just
separate extension problems. Write $f=(f_n):eJ\to S^{(\mathbb N)}$.
Each $f_n$ is nonzero and the family is pointwise finite on $eJ$.
Right self-injectivity extends each coordinate to
$h_n:eR\to R$, but $h_n(e)\ne0$ for every $n$.
Thus no such coordinate extensions form a pointwise finite
family on $eR$. The actual extension into $I$ exists, whereas
an extension into $E$ does not.
The second-layer test therefore narrows the remaining task;
it does not turn $R/J^2$ into a self-injective ring or permit
an application of a radical-square-zero self-injective theorem.
\end{remark}

\begin{corollary}\label{support-criterion}
Under \eqref{F}, the following are equivalent:
$R$ is QF; $\rho_I$ is open; every cyclic submodule of $I$
has finite socle; and
\begin{equation}\label{remaining}
\text{for every }x\in I\text{ there is a finite }F\subseteq\N
\text{ such that }xR\cap S^{(\N)}\subseteq S^{(F)}.
\end{equation}
\end{corollary}
\begin{proof}
Combine Theorems~\ref{quotient} and \ref{projectivity}.
Since $\Soc I=S^{(\N)}$, the intersection in
\eqref{remaining} is $\Soc(xR)$. A finite-length submodule
of $S^{(\N)}$ has finite generators and therefore a common
finite support. Conversely, $S^{(F)}$ has finite length.
\end{proof}

\subsection{The second socle and the opposite extension problem}
Put $V=J/J^2$ and $L_2=\la(J^2)=\ra(J^2)$, the second
socle of either regular module. Here $aJ^2=0$ implies
$aJ\subseteq S$, then $JaJ=0$, $Ja\subseteq S$ and
$J^2a=0$; the converse is symmetric.
The following is the second-layer case of the classical
annihilator-duality calculation; see Nicholson--Yousif
\cite[proof of Proposition 4]{NY97}.

\begin{proposition}\label{second-layer-duality}
Multiplication induces an isomorphism of $D$-$D$-bimodules
\[
 \Phi:L_2/S\xrightarrow{\ \sim\ }
 V^*=\Hom_D(V_D,S_D),\qquad
 \Phi(a+S)(j+J^2)=aj.
\]
On the dual the actions are
$(d f)(v)=d f(v)$ and $(f d)(v)=f(dv)$.
The opposite multiplication map
\[
 \varepsilon:V\longrightarrow
 \Hom_{D^{\mathrm{op}}}(L_2/S,S),\qquad
 \varepsilon(b+J^2)(a+S)=ab,
\]
is injective.
\end{proposition}
\begin{proof}
If $aJ^2=0$, then $aJ\subseteq S$.
Also $JL_2\subseteq S$, because $j(aJ)=0$ for $j\in J$.
Thus $L_2/S$ is a $D$-$D$-bimodule, and the displayed actions
make $\Phi$ a bimodule homomorphism with kernel
$\la(J)/S=0$.
Any right $D$-map $V\to S$ gives a map $J\to R$ and extends,
by right self-injectivity, to left multiplication on $R$.
Its multiplier belongs to $L_2$, proving surjectivity.
Finally the kernel of the opposite map is
$\ra(L_2)/J^2=0$, by the double-annihilator identity.
\end{proof}

The first dual is therefore fully represented by multiplication.
The opposite extension is a separate requirement, not another
consequence of right self-injectivity.

\begin{proposition}\label{opposite-second-layer-test}
Under \eqref{F}, the following are equivalent:
\begin{enumerate}
\item $R$ is QF;
\item every left $R$-linear map $L_2\to S$ killing $S$
extends to the left regular module;
\item the extension property in \textup{(ii)} holds for maps
with simple image.
\end{enumerate}
If $R$ is not QF, the failure already occurs on a countably
generated left ideal $L\subseteq L_2$ containing $S$.
The offending map on $L$ extends from every finitely generated
left subideal of $L$, with explicit multipliers, but not from $L$.
\end{proposition}
\begin{proof}
Left self-injectivity gives (i)$\Rightarrow$(ii), and
(ii)$\Rightarrow$(iii) is immediate.
Suppose $R$ is not QF and take a quotient
$q:V\twoheadrightarrow T^{(\N)}$, as in
Theorem~\ref{homogeneous}. Fix an embedding $\tau:T\hookrightarrow S$
and $0\ne t\in T$; put $s_0=\tau(t)$.
Let $f_n=\tau q_n$ be the coordinate maps, regarded as maps on $J$.
They are pointwise finite. Therefore
\[
f_\infty:J\longrightarrow S,\qquad
f_\infty(j)=\sum_n f_n(j)
\]
is well-defined and right $R$-linear.
Right self-injectivity represents $f_n$ and $f_\infty$ by
left multiplication with $a_n,a_\infty\in L_2$.
Choose $j_m\in J$ whose image under $q$ has $t$ in coordinate
$m$ and zero elsewhere. Then
\begin{equation}\label{biorthogonal-multipliers}
a_nj_m=\delta_{nm}s_0,\qquad a_\infty j_m=s_0,
\qquad a_\infty j=\sum_n a_nj\quad(j\in J).
\end{equation}
The last sum is finite for each fixed $j$; it is an identity
of maps on $J$.

Set $L=S+Ra_\infty+\sum_nRa_n$ and define
\[
\phi\left(s+\sum_n r_na_n+r_\infty a_\infty\right)=r_\infty s_0,
\]
where the sum is finite. This is well-defined: multiply any
zero relation on the right by $j_m$, choosing $m$ outside
the finite support, to obtain $r_\infty s_0=0$.
It is left linear and has simple image $Rs_0$.
Indeed, if $u=rs_0\ne0$, left multiplication by $r$ restricts
to a nonzero map from the simple right ideal $s_0R$ onto
$uR$, hence to an isomorphism. Its inverse extends to an
endomorphism of $R_R$ by right self-injectivity. That
endomorphism is left multiplication by some $c\in R$,
so $cu=s_0$. Every nonzero left submodule of $Rs_0$ therefore
contains $s_0$, and $Rs_0$ is simple.
The ideal $L$ is countably generated because ${}_RS$ is finite.

For $L_N=S+Ra_\infty+\sum_{n\leq N}Ra_n$, right multiplication
by any $j_m$ with $m>N$ extends $\phi|_{L_N}$ to $R$.
Every finitely generated left subideal of $L$ lies in some $L_N$.
If $\phi$ extended from $L$ by right multiplication with $b$,
then $Sb=0$, so $b\in J$. The equations $a_nb=0$ for all $n$
would give $q(b+J^2)=0$ and hence $a_\infty b=0$, contradicting
$\phi(a_\infty)=s_0$.
Finally $\phi$ factors through the left $D$-submodule $L/S$
of $L_2/S$. Semisimplicity extends it to $L_2/S$ with values
in $Rs_0$. This produces a simple-image map in (iii) which
cannot extend to $R$, and proves the equivalence.
\end{proof}

Choose left generators $s_1,\ldots,s_r$ of $S$.
The proof exhibits the actual one-unknown system
\begin{equation}\label{left-compactness-system}
s_i b=0\ (1\leq i\leq r),\qquad
a_n b=0\ (n\geq1),\qquad a_\infty b=s_0.
\end{equation}
Every finite subsystem is solved by some $j_m\in J$, whereas
the whole system has no solution in $R$.
These are equations with coefficients on the left of the
unknown, so pure-injectivity of the \emph{right} regular module
does not solve them.

\begin{remark}\label{second-layer-continuity}
Give $V^*=\Hom_D(V,S)$ its pointwise topology with $S$ discrete.
Its continuous left $D$-linear maps to $S$ are exactly the
evaluations at elements of $V$. Indeed, continuity gives a
factorization through the dual of a finite-length submodule
of $V$. By Lemma~\ref{structure}, $S$ induces finite Morita
duality, so evaluation identifies that finite bidual with
the submodule itself.
For any extension of $\phi$ to $L_2/S$, the associated map on
$V^*$ is discontinuous. The sequence
\[
\Phi\left(a_\infty-\sum_{n\leq N}a_n+S\right)
\]
converges pointwise to zero by \eqref{biorthogonal-multipliers},
but its image is constantly $s_0$.
This topology is not the finite topology on $\End_R(R)$.
Automatic continuity here would already supply the missing
opposite-side extension; it has not been deduced from \eqref{F}.
\end{remark}

\subsection{Compatible radical actions}
\begin{theorem}\label{automorphism-criterion}\label{action-criterion}
Assume \eqref{F}. If the left and right actions of $J$ on $L_2$
give the same set of additive endomorphisms, that is,
\begin{equation}\label{matching-actions}
\{a\mapsto ja:j\in J\}=\{a\mapsto aj:j\in J\}
\quad\text{in }\End_{\mathbb Z}(L_2),
\end{equation}
then $R$ is QF. In particular, the conclusion holds if there
is a ring automorphism $\sigma$ such that
\begin{equation}\label{automorphism-identity}
aj=\sigma(j)a\qquad(a\in L_2,\ j\in J).
\end{equation}
The latter condition holds if every element
of $L_2$ commutes with every element of $J$.
\end{theorem}
\begin{proof}
If $R$ were not QF, take the multipliers in
\eqref{biorthogonal-multipliers}. On the right ideal
$L'=S+a_\infty R+\sum_n a_nR$ define
\[
\psi\left(s+\sum_n a_nr_n+a_\infty r_\infty\right)=s_0r_\infty.
\]
To check well-definedness, choose $c_m\in J$ with
$c_ma=aj_m$ for all $a\in L_2$, using \eqref{matching-actions}.
Multiply a zero relation on the left by $c_m$, with $m$
outside its finite support. Then $c_mS=0$,
$c_ma_n=a_nj_m$, and $c_ma_\infty=a_\infty j_m$.
Thus $s_0r_\infty=0$ as required.
Right self-injectivity extends $\psi$ to left multiplication
by $b\in R$. Since $bS=0$, we have $b\in J$.
Again by \eqref{matching-actions}, choose $c\in J$ such that
$ba=ac$ for all $a\in L_2$.
For every $n$, $0=ba_n=a_nc$, so $q(c+J^2)=0$.
Hence $ba_\infty=a_\infty c=0$, contradicting
$\psi(a_\infty)=s_0$.
Finally, \eqref{automorphism-identity} implies
\eqref{matching-actions} because $\sigma(J)=J$.
\end{proof}

The containment of right multiplications by $J$ in the left
multiplications by $J$ on $L_2$ is already equivalent to
$[R,J]\subseteq J^2$. Indeed, apply the identity $aj=ba$
also to $ar\in L_2$. Comparing $(ar)j=bar$ with
$(aj)r=bar$ gives $a(rj-jr)=0$ for all $a\in L_2$, whence
$rj-jr\in\ra(L_2)=J^2$. Conversely, this commutator condition
makes right multiplication by $j$ on $L_2$ right $R$-linear.
Right self-injectivity extends it to left multiplication by
$b\in R$, and $bS=0$ implies $b\in J$.
Thus \eqref{matching-actions} includes a strong restriction
on the first radical layer. In fact the following dimension
argument gives QF under that weaker restriction. The preceding
proof records a direct transport of the particular Baer
obstruction, rather than a new class beyond this dimension test.

\begin{proposition}\label{central-layer-criterion}
Assume \eqref{F}. If $[R,J]\subseteq J^2$, then $R$ is QF.
\end{proposition}
\begin{proof}
The $D$-$D$-bimodule $V=J/J^2$ is central in the strong
sense $dv=vd$ for every $d\in D$, $v\in V$.
In this proof, the length of a semisimple module means the
cardinal number of its simple summands.
The ideal generated by commutators of $D$ annihilates $V$.
Consequently $V$ is supported only on the simple blocks of
$D$ which are fields; central idempotents exclude cross terms.
If $R$ were not QF, the right semisimple length
$\kappa=\ell_D(V)$ would be infinite. Each tensor power
$V^{\otimes_D n}$ then has left length at most $\kappa$,
and multiplication maps it onto $J^n/J^{n+1}$.
The left semisimple module $L_2/S$ has the finite filtration
$((L_2\cap J^n)+S)/S$; each factor is a subquotient of
$J^n/J^{n+1}$ (a subquotient of $D$ when $n=0$).
Thus its left length is at most $\kappa$.

On the other hand, some simple right type $T$ has multiplicity
$\kappa$ in $V$. By Proposition~\ref{second-layer-duality},
$L_2/S$ has a direct summand $K^\kappa$, where
$K=\Hom_D(T,S)$ is a nonzero simple left $D$-module.
Morita equivalence for its simple block and the
Erd\H{o}s--Kaplansky theorem give
$\ell_D(K^\kappa)>\kappa$, a contradiction.
Here the dimension theorem is used in its classical form
$\dim_F F^\kappa=|F^\kappa|>\kappa$ for a division ring $F$
and an infinite cardinal $\kappa$; see
\cite[Section 6]{BergmanEK}.
\end{proof}

\begin{example}\label{quantum-example}
For a field $k$ and $0\ne\lambda\in k$, consider the familiar algebra
\[
R=k\langle x,y\rangle/(x^2,y^2,yx-\lambda xy).
\]
It has basis $1,x,y,xy$, $S=kxy$ and $L_2=J=(x,y)$.
The coefficient of $xy$ in a product is a nondegenerate
associative bilinear form, so $R$ is Frobenius.
The automorphism $\sigma(x)=\lambda x$,
$\sigma(y)=\lambda^{-1}y$ satisfies
\eqref{automorphism-identity}, as is checked on the basis of $J$.
When $\lambda\ne1$, $L_2$ and $J$ do not commute.
Thus the twisted identity allows noncommuting radical actions.
The condition is illustrated here by a familiar Frobenius algebra.
The condition is not necessary for a QF ring: for
$R=M_2(k[t]/(t^2))$ one has $L_2=R$. Evaluating an equality
$a j=c a$ at $a=1$ forces $c=j$, so \eqref{matching-actions}
would make every element of $J$ central, which is false.
Thus it cannot be imposed on arbitrary rings satisfying \eqref{F}.
\end{example}

\subsection{A product map and Chase's correction term}
A $p$-functor is an additive subfunctor of the underlying
abelian-group functor on right modules which commutes with
products. Such functors also commute with direct sums.
Chase's lemma for a map $g:\prod_m X_m\to\bigoplus_\lambda Y_\lambda$
and a descending chain $P_n$ allows a finite exceptional
subsum together with the correction
$\bigcap_nP_n(\bigoplus_\lambda Y_\lambda)$; see
\cite[Lemma 11 and its proof]{HZ}.
For the maps arising from the preceding construction, that
correction already contains the entire image after restriction.

\begin{proposition}\label{socle-chase-boundary}
Let $S$ be a two-sided ideal of any ring $R$, of finite length
as a left module. Let $P_1\supseteq P_2\supseteq\cdots$ be
$p$-functors on right modules. If
$g:X\to R^{(\Lambda)}$ has image in $S^{(\Lambda)}$, then
for some $n_0$,
\[
g(P_{n_0}X)\subseteq\bigcap_{n\geq1}P_n(R^{(\Lambda)}).
\]
\end{proposition}
\begin{proof}
Each $P_nR$ is invariant under left multiplication, an
endomorphism of $R_R$. Hence $H_n=P_nR\cap S$ is a left
submodule of $S$. Choose $n_0$ with $H_{n_0}=\bigcap_nH_n$.
Naturality and commutation with sums give
\[
g(P_{n_0}X)\subseteq P_{n_0}(R^{(\Lambda)})\cap S^{(\Lambda)}
=H_{n_0}^{(\Lambda)}
\subseteq\bigcap_nP_n(R^{(\Lambda)}).
\]
\end{proof}

In the non-QF situation of Theorem~\ref{homogeneous}, extend
the map on $eJ$ to $f:J\to T^{(\N)}\subseteq S^{(\N)}$
by $f(j)=f(ej)$. It remains surjective.
The coefficient $J_{R/S}$ is injective by
Theorem~\ref{descent}. Thus the sum map
$\sigma:J^{(\N)}\to J$ extends to $h:J^{\N}\to J$ over $R/S$,
and hence over $R$. Consequently
\[
g=fh:J^{\N}\longrightarrow E
\]
is a product-to-sum map. Its restriction to each input
coordinate is $f$, since $h$ restricts to the identity there.
Proposition~\ref{socle-chase-boundary} applies to every descending
$p$-functor chain: no finite exceptional output coordinates
are needed, because the correction term contains the whole
restricted image. A use of Chase's lemma on this map alone
therefore does not force finite support.

\subsection{The remaining finite-support problem}
\begin{problem}\label{remaining-gap}\label{cyclic-problem}
Under \eqref{F}, prove \eqref{remaining}; equivalently,
exclude the cyclic essential extension in
\eqref{homogeneous-witness}.
\end{problem}

Every individual element of $xR\cap S^{(\N)}$ already has
finite support. What is missing is one finite set of
coordinates valid for the whole socle of $xR$.
Neither the all-degree nonvanishing in
Theorem~\ref{persistence} nor the descent in
Theorem~\ref{descent} provides this uniform bound.
Theorem~\ref{action-criterion} supplies a sufficient compatibility
condition. Example~\ref{quantum-example} shows that the condition
is not automatic under \eqref{F}, so it cannot serve as the
missing assertion in full generality.
If the uniform bound were established, Theorem~\ref{projectivity} would
construct finite projective-injective summands covering
$I$, and essentiality would then give $I=E$.

The finite topology on $\End_R(I)$ should also be kept
separate from the topology induced on $I$ by the product
of discrete copies of $R$. For $R\ne0$, choose $0\ne s\in S$.
The finite truncations of $(s,s,\ldots)$ lie in $E$, but
their product limit does not lie in $I$, since
$\Soc I=S^{(\N)}$. Thus $I$ is not complete for that induced
topology. This has no bearing on the completeness of its
endomorphism ring in the finite topology.

\begin{remark}\label{cancellation-boundary}
The cancellation theorem cannot be applied to the fixed hull
by replacing finite cogeneration with direct finiteness.
Over a semiprimary ring, a quasi-injective module $M$ is
directly finite exactly when its essential socle has finite
length. Indeed, if some simple type has infinite multiplicity,
the socle has an isomorphic proper summand. Its inclusion
extends to a monic endomorphism of $M$, which splits by
Lemma~\ref{relative-extension} and has proper image.
Conversely, finite socle length excludes a proper isomorphic
summand. There are only finitely many simple types.
In particular $I\cong I\oplus I$, and the homogeneous $H$
in Theorem~\ref{homogeneous} also satisfies $H\cong H\oplus H$,
by uniqueness of injective hulls of their socles.
Proposition~\ref{finite-blocks} further shows that deleting
any finite set of coordinates leaves $B$ unchanged.
These facts explain the limitation of finite cancellation;
they do not prove that $B$ is nonzero or exclude that possibility.
\end{remark}

\section{Products and injective hulls over a square-zero ring}\label{nonopen-section}
We now apply the countable hull construction to the same ring as in
Theorem~\ref{counterexample}. The coefficient module will be the
injective hull of its simple module, rather than the regular module.
The resulting hull quotient can be computed completely. Moreover,
the submodule used in the product counterexample is its intersection
with the original product inside a common ambient module.

Retain $R$, $J$, $L=R/J$ and $P=\prod_nE_n$ from
Section~\ref{counterexample-section}, and identify $J$ with
$V=k^{(\N)}$ using the basis $x_n$. Thus $R=k\oplus V$ with
\begin{equation}\label{square-ring}
(c,v)(b,w)=(cb,cw+bv).
\end{equation}
Write $\mathcal H=\End_k(V)$ and let $\mathcal F\subseteq\mathcal H$
be the two-sided ideal of finite-rank maps. Put
\[
C=\Hom_k(R,k),\qquad E_C=C^{(\N)},\qquad
I_C=\Hom_k(R,V).
\]
The right action on these coinduced modules is $(fr)(a)=f(ra)$.
Evaluation at $1$ and restriction to $V$ give
\begin{equation}\label{coinduced-action}
I_C=V\oplus\mathcal H,\qquad
(z,h)(c,v)=(cz+h(v),ch).
\end{equation}
Coordinate evaluation embeds $I_C$ in $C^{\N}$ and identifies
$E_C$ with its usual finite-coordinate submodule.

\begin{theorem}\label{calkin-hull}
With these inclusions, $C$ is a finitely cogenerated injective
module with socle $L$, and $I_C$ is the injective hull of $E_C$.
More explicitly,
\begin{equation}\label{calkin-hull-formula}
E_C=V\oplus\mathcal F\subseteq V\oplus\mathcal H=I_C,
\qquad I_C/E_C\cong\mathcal H/\mathcal F.
\end{equation}
The quotient is annihilated by $J$. Pullback of
$0\to E_C\to I_C\to I_C/E_C\to0$ gives an isomorphism
\begin{equation}\label{calkin-ext}
\mathcal H/\mathcal F\xrightarrow{\ \sim\ }
\Ext_R^1(L,E_C),\qquad h+\mathcal F\longmapsto\xi_h,
\end{equation}
where $\xi_h$ is also the pushout of $0\to J\to R\to L\to0$
along the map $v\mapsto(h(v),0)$ into $\Soc E_C$.
For $a\in\mathcal H$, the formula
\[
\widehat a(z,h)=(az,a\circ h)
\]
defines an endomorphism of $I_C$ preserving $E_C$, and
\begin{equation}\label{universal-baer}
(\widehat a|_{E_C})_*\xi_{\id_V}=\xi_a.
\end{equation}
Thus, under this specified action, $\Ext_R^1(L,E_C)$ is the
left regular module of $\mathcal H/\mathcal F$.
\end{theorem}
\begin{proof}
For any $k$-vector space $W$, the adjunction
\[
\Hom_R(X,\Hom_k(R,W))\cong\Hom_k(X,W)
\]
shows that $\Hom_k(R,W)$ is injective. Its socle consists
of maps vanishing on $J$, and is isomorphic to $W$.
It is essential: if $f|_J\ne0$, multiplication by a suitable
$v\in J$ gives a nonzero socle element. In particular,
$C$ has simple essential socle and is finitely cogenerated.

An element of $E_C$ has only finitely many nonzero coordinate
maps $R\to k$. In \eqref{coinduced-action} this says that
$z\in V$ and $h(V)$ is contained in a fixed finite-coordinate
subspace of $V$. This condition on $h$ is equivalent to
finite rank: each finite-dimensional subspace of $V=k^{(\N)}$
is contained in such a coordinate subspace.
This proves the first equality in \eqref{calkin-hull-formula}.
Since $\Soc I_C=V\oplus0=\Soc E_C$ is essential,
$I_C$ is the injective hull of $E_C$. Formula
\eqref{coinduced-action} shows that its quotient is killed by $J$.

Apply $\Hom_R(L,-)$ to the hull sequence. Since $\Soc I_C\subseteq E_C$ and $I_C$ is injective, the
relevant part of the long exact sequence is
\[
0\longrightarrow\Hom_R(L,I_C/E_C)
\xrightarrow{\partial}\Ext_R^1(L,E_C)\longrightarrow0.
\]
Thus the connecting map is an isomorphism. Identifying
$\Hom_R(L,I_C/E_C)$ with $I_C/E_C$ by evaluation at $1+J$
gives \eqref{calkin-ext}. The lift
$R\to I_C$, $1\mapsto(0,h)$ restricts on $J$ to
$v\mapsto(h(v),0)$, proving the stated pushout description.

The formula for $\widehat a$ respects \eqref{coinduced-action}
and preserves $V\oplus\mathcal F$. Its action on the quotient
is left composition by $a$. Naturality of the connecting map
gives \eqref{universal-baer}. Its kernel as an action of
$\mathcal H$ is exactly $\mathcal F$, as is also seen by
applying the action to $\xi_{\id_V}$.
\end{proof}

\begin{theorem}[A counterexample with free cyclic factors]\label{free-product-bridge}
For the same ring, put $P_{\mathrm f}=R^{\N}$.
Relative to the basis $x_n$ of $V$, there are isomorphisms
of right $R$-modules
\begin{equation}\label{full-product-defect}
\frac{\Soc P_{\mathrm f}}{P_{\mathrm f}J}
 \cong\frac{\End_k(V)}{\mathcal F}
 \cong\frac{I_C}{E_C}
 \cong\Ext_R^1(L,E_C).
\end{equation}
There is also a noncanonical decomposition
\begin{equation}\label{free-product-decomposition}
P_{\mathrm f}\cong(R\otimes_k k^{\N})\oplus(I_C/E_C).
\end{equation}
In particular $P_{\mathrm f}\cong L\oplus P_{\mathrm f}$, whereas,
for distinct nonnegative integers $a,b$, all finite
$I_0,J_0\subseteq\N$, and every $Y\mid R^{I_0}$,
\begin{equation}\label{free-finite-failure}
L^a\oplus Y\not\cong L^b\oplus R^{J_0}.
\end{equation}
Consequently the product comparison of Question 3.1 also
fails when every family factor $C_n=D_n=R$ is free cyclic.
These family factors are not finitely cogenerated.
For a nonzero $h\in\End_k(V)$, let
$z_h=(h(x_n))_n\in\Soc P_{\mathrm f}$ and $L_h=kz_h$.
The specified inclusion $L_h\hookrightarrow P_{\mathrm f}$
splits if and only if $\xi_h\ne0$, equivalently if $h$
has infinite rank.
\end{theorem}
\begin{proof}
Write $P_{\mathrm f}=k^{\N}\oplus V^{\N}$ as a vector space.
Its action is
\[
(a,(v_n))(c,v)=(ca,(cv_n+a_nv)_n).
\]
Thus $\Soc P_{\mathrm f}=0\oplus V^{\N}$, and
$P_{\mathrm f}J=0\oplus\mathcal A$, where
\[
\mathcal A=\{(v_n)\in V^{\N}:
             \dim_k\langle v_n:n\in\N\rangle<\infty\}.
\]
Indeed, a finite sum of products with elements of $J$ has
this property. Conversely, writing all $v_n$ in a fixed
finite basis expresses the sequence as such a finite sum.
The vector-space isomorphism
\[
\Theta:V^{\N}\longrightarrow\End_k(V),\qquad
\Theta((v_n))(x_n)=v_n,
\]
sends $\mathcal A$ exactly to $\mathcal F$.
This proves \eqref{full-product-defect} using
Theorem~\ref{calkin-hull}. All displayed modules are killed
by $J$, so the identifications are also $R$-linear.

Choose $V^{\N}=\mathcal A\oplus W$ as vector spaces. Then
\[
P_{\mathrm f}=(k^{\N}\oplus\mathcal A)\oplus(0\oplus W)
\]
is a decomposition of $R$-modules. Here
$W\cong V^{\N}/\mathcal A\cong I_C/E_C$, and the map
$R\otimes_k k^{\N}\to k^{\N}\oplus\mathcal A$ sending
$(c,v)\otimes a$ to $(ca,(a_nv)_n)$ is an isomorphism:
the displayed action shows that it is $R$-linear, with $R$
acting on the first tensor factor. Its radical components map
$V\otimes_k k^{\N}$ onto
$\mathcal A$. For injectivity, write a tensor as
$\sum_{i=1}^r v_i\otimes a_i$ with the $v_i$ linearly
independent. If every coordinate $\sum_i a_i(n)v_i$ is zero,
then all $a_i(n)$ vanish, so the tensor is zero.
This proves \eqref{free-product-decomposition}.
The space $W$ is infinite dimensional: diagonal operators
with disjoint infinite supports
are linearly independent modulo $\mathcal F$.
Hence $W\cong L\oplus W$, which proves absorption of $L$.

Every finite free module satisfies
$\Soc(R^m)=J^{\oplus m}=R^mJ$. Intersecting this equality with a
direct summand $Y$ gives $\Soc Y=YJ$, as in the proof of
Theorem~\ref{counterexample}. It follows that
\[
\Soc(L^a\oplus Y)/(L^a\oplus Y)J\cong L^a,
\qquad
\Soc(L^b\oplus R^{J_0})/(L^b\oplus R^{J_0})J\cong L^b.
\]
An isomorphism in \eqref{free-finite-failure} would identify
these finite-dimensional quotients, a contradiction. This
argument does not assume the finite free modules have finite
length. Taking $A=L$, $B=L^2$ proves the free-family claim.
Finally, every map $P_{\mathrm f}\to L$ kills
$P_{\mathrm f}J$. Thus the specified inclusion of $L_h$
cannot split if $z_h\in P_{\mathrm f}J$. Otherwise a
$k$-linear functional $\lambda:P_{\mathrm f}/P_{\mathrm f}J\to k$
taking the value $1$ on $z_h+P_{\mathrm f}J$ gives a map
$r:P_{\mathrm f}\to L_h$ by $r(x)=\lambda(x+P_{\mathrm f}J)z_h$.
Both the quotient and $L_h$ are annihilated by $J$, so $r$
is $R$-linear, and $r(z_h)=z_h$. The description of $\mathcal A$
and \eqref{calkin-ext} give the asserted equivalences.
\end{proof}

\begin{theorem}\label{product-hull-bridge}
Embed $E_n$ in $C=k\oplus V^*$ by
\[
u_n\longmapsto(0,x_n^*),\qquad v_n\longmapsto(1,0),
\]
where $x_n^*(x_m)=\delta_{nm}$, and regard $P$ as a submodule
of $C^{\N}$. In the coordinates $P=F\oplus F$ of
\eqref{product-action}, where $F=k^{\N}$ and $F_0=k^{(\N)}$,
\begin{equation}\label{bridge-intersections}
P\cap I_C=F\oplus F_0,\qquad
P\cap E_C=F_0\oplus F_0=\bigoplus_n E_n.
\end{equation}
Inclusion in $I_C$ therefore induces an injection
\begin{equation}\label{bridge-map}
\frac{\Soc P}{PJ}\cong\frac{F}{F_0}
 \cong\frac{P\cap I_C}{P\cap E_C}
 \lhook\joinrel\longrightarrow\frac{I_C}{E_C}
 \cong\Ext_R^1(L,E_C).
\end{equation}
With the specified coordinates this map sends
$\lambda+F_0$ to $D_\lambda+\mathcal F$, and then to
$\xi_{D_\lambda}$, where $D_\lambda(x_n)=\lambda_nx_n$.
The diagonal map $F/F_0\to\mathcal H/\mathcal F$ is a unital
$k$-algebra embedding. The Ext identification in
\eqref{bridge-map} is one of modules, not of Yoneda rings.
Finally,
\begin{equation}\label{bridge-dimension}
\dim_k\Ext_R^1(L,E_C)=\dim_k(F/F_0)=|k|^{\aleph_0}.
\end{equation}
\end{theorem}
\begin{proof}
The indicated embeddings respect multiplication by every $x_m$.
An element $(a,b)\in P$ corresponds to the map
\[
R\longrightarrow k^{\N},\qquad
\left(c,\sum_n d_nx_n\right)\longmapsto
cb+(a_nd_n)_n.
\]
Its image lies in $V=F_0$ precisely when $b\in F_0$;
there is no restriction on $a$, since $(d_n)$ has finite
support. The map has only finitely many nonzero output
coordinates precisely when both $a$ and $b$ have finite
support. This proves \eqref{bridge-intersections}.
Modulo $E_C$ its value is $D_a+\mathcal F$.
For completeness, the two occurrences of $F/F_0$ in
\eqref{bridge-map} can be related without choosing a complement
to $F_0$. The coordinate maps $E_n\to\Soc E_n$,
$u_n\mapsto v_n$, $v_n\mapsto0$, are $R$-linear and induce
$d:P\to\Soc P$, $d(a,b)=(0,a)$. On $P\cap I_C$ this map
is onto $\Soc P$, its kernel $0\oplus F_0$ lies in $P\cap E_C$,
and $d(P\cap E_C)=PJ$. It therefore gives the middle
isomorphism in \eqref{bridge-map}, with the indicated direction
reversed. All identifications are relative to the displayed
coordinate embeddings. Notice that $D_\lambda$ has finite
rank exactly when $\lambda\in F_0$, and
$D_\lambda D_\mu=D_{\lambda\mu}$. This proves the algebra
embedding assertion; it does not assign a Yoneda product to
$\Ext^1$.

The Erd\H{o}s--Kaplansky theorem gives
$\dim_kF=|k|^{\aleph_0}>\aleph_0$
\cite[Section 6]{BergmanEK}; quotienting by $F_0$ leaves that
dimension unchanged. The injection gives the lower bound
in \eqref{bridge-dimension}. Since an endomorphism of $V$
is determined by the images of its countable basis,
$|\mathcal H|\leq|V|^{\aleph_0}=|k|^{\aleph_0}$.
This gives the upper bound.
\end{proof}

\begin{corollary}\label{diagonal-baer-witnesses}
For $\lambda\in F$ with infinite support, put
$x_\lambda=(0,D_\lambda)\in I_C$ and $U_\lambda=x_\lambda R$.
Then
\[
U_\lambda/U_\lambda J\cong L,\qquad
U_\lambda J=\Soc U_\lambda=U_\lambda\cap E_C
 =\im D_\lambda\oplus0\cong L^{(\N)},\qquad U_\lambda J^2=0.
\]
The sequence
\[
0\longrightarrow E_C\longrightarrow E_C+U_\lambda
 \longrightarrow L\longrightarrow0
\]
represents $\xi_{D_\lambda}$ and is nonsplit, although its
pushout to every finite coordinate subsum $C^A$ is split.
Two such sequences, with the displayed kernel and quotient
identifications, are equivalent if and only if their
sequences $\lambda$ differ by an element of $F_0$.
\end{corollary}
\begin{proof}
Formula \eqref{coinduced-action} gives
$U_\lambda=(\im D_\lambda\oplus0)+k(0,D_\lambda)$.
Since $D_\lambda$ has infinite rank, no nonzero scalar
multiple of it belongs to $\mathcal F$; multiplication by
$J$ gives exactly $\im D_\lambda\oplus0$.
The quotient $(E_C+U_\lambda)/E_C$ is identified with $L$
by sending $(0,D_\lambda)+E_C$ to $1+J$. With this
identification its extension class is $\xi_{D_\lambda}$.
Two such extensions are equivalent precisely when
$D_\lambda-D_\mu$ has finite rank, or equivalently
$\lambda-\mu\in F_0$.
Every $C^A$ with $A$ finite is injective, so the indicated
finite-coordinate pushout splits.
\end{proof}

The intersection $P\cap I_C=F\oplus F_0$ is the summand
used in the proof of Theorem~\ref{counterexample}; this gives
the map \eqref{bridge-map}. Equations \eqref{calkin-ext} and
\eqref{universal-baer} describe all extensions of $L$ by
$E_C$ with fixed kernel and quotient identifications. The usual fact that a sum of injectives need not be
injective, and the finite-coordinate splitting assertion
alone, are classical consequences of the theory of
$\Sigma$-injectivity \cite[Corollary 7]{HZ}.
The quotient $\mathcal H/\mathcal F$ is the classical
algebraic Calkin algebra; compare its definition in
\cite[Definition 3.3]{Efimov}. Bergman already uses this
finite-rank quotient in examples concerning infinite products
\cite[Section 2]{Bergman2006}. Those examples form quotients
of an endomorphism ring; here the base ring remains the
square-zero ring $R$. The additional conclusions are the full
defect identification \eqref{full-product-defect}, its realization
as a semisimple summand in \eqref{free-product-decomposition},
the free-family failure of finite comparison, and the intersection
map \eqref{bridge-map} for the length-two construction.
The action of one Baer class describes the associated extensions.

There is also a useful distinction between the two comparison
problems on this ring. The product $C^{\N}$ is injective,
so Theorem~\ref{product-comparison} applies to decompositions
of it with finitely cogenerated distinguished and family
factors as specified there. Nevertheless $E_C$ is not injective,
by \eqref{calkin-hull-formula}. Finite comparison in that
injective product therefore does not imply a uniform finite
support assertion for extensions into its countable sum.

The coefficient also lacks the scalar endomorphism property
of the regular module in \eqref{F} and of the coefficient
in Theorem~\ref{descent}. Indeed $C=k\oplus V^*$, and
commutation with multiplication by $v\in V$ gives
\[
\End_R(C)\cong k\ltimes V^{**},\qquad
(z,f)\longmapsto(cz+\beta(f),cf),
\quad c\in k,\ \beta\in V^{**}.
\]
Under this identification the scalar map $R\to\End_R(C)$
is the evaluation inclusion $k\oplus V\to k\oplus V^{**}$,
which is not onto. A nonzero functional on $F/F_0$ gives
an element of $V^{**}$ vanishing on $F_0\subseteq V^*=F$;
no nonzero evaluation by an element of $V$ does so.
Thus finite cogeneration of an injective coefficient alone
does not impose the endomorphism constraints present in the
Faith problem.

\subsection{Socle restriction in the explicit hull}
The same hull also realizes the topological obstruction
considered in Theorem~\ref{restriction}.
\begin{proposition}\label{nonopen-example}
The endomorphisms of $I_C=V\oplus\mathcal H$ are precisely
\begin{equation}\label{endomorphism-formula}
F_{a,\beta}(z,h)=(az+\beta(h),a\circ h),\qquad
a\in\End_k(V),\quad\beta\in\Hom_k(\mathcal H,V).
\end{equation}
Restriction to $\Soc I_C=V\oplus0$ is the continuous
surjective ring map $\rho:F_{a,\beta}\mapsto a$.
It has the algebraic ring section $a\mapsto F_{a,0}$ and
square-zero kernel $\Jac\End_R(I_C)$. Its quotient topology
is discrete, whereas the finite topology on $\End_k(V)$
is not discrete. Thus $\rho$ is not open and this section
is not continuous. Both endomorphism rings are complete
and Hausdorff in their finite topologies.
\end{proposition}
\begin{proof}
An endomorphism is $k$-linear and preserves $V\oplus0$, so
it has the form $(z,h)\mapsto(az+\beta(h),\gamma(h))$.
Commutation with multiplication by $v\in V$ gives
$\gamma(h)(v)=a(h(v))$. Hence $\gamma(h)=a\circ h$;
conversely, every displayed map respects
\eqref{coinduced-action}. Composition is
\[
F_{a,\beta}F_{b,\delta}=F_{ab,\,a\delta+\beta\lambda_b},
\qquad\lambda_b(h)=b\circ h.
\]
This proves the ring section and square-zero assertions.
The radical equality and continuity follow from
Theorem~\ref{restriction}.
For $x=(0,\id_V)$ we have
\[
F_{a,\beta}(x)=(\beta(\id_V),a),\qquad
\rho\bigl(\mathcal N_{I_C}(xR)\bigr)=0.
\]
Thus the kernel is open. Every finite subset of $V$ is
annihilated by a nonzero endomorphism, so zero is not open
in the target finite topology. The inverse image of
$\mathcal N_{I_C}(xR)$ under the displayed section is zero,
which also proves its discontinuity. Completeness was
noted before Theorem~\ref{restriction}.
\end{proof}

For $\lambda=(1,1,\ldots)$, the cyclic module $U_\lambda$
above is a copy of $R$. This ring is not self-injective:
for independent $v_1,v_2\in V$, the map $kv_1\to R$,
$v_1\mapsto v_2$, is $R$-linear but cannot be multiplication
by an element of $R$. This is an application of the countable construction with
injective coefficient $C$, outside hypothesis \eqref{F}.

\clearpage
\begingroup\small\setlength{\bibsep}{1pt}

\endgroup
\end{document}